\documentclass[11pt]{article}
\usepackage[T1]{fontenc}
\usepackage{lmodern}
\usepackage{amsmath,amssymb,amsthm,mathtools,mathrsfs}
\usepackage[margin=1.0in]{geometry}
\usepackage{booktabs,tabularx,array,longtable}
\usepackage{microtype}
\usepackage{xcolor}
\usepackage[colorlinks=true,linkcolor=blue!50!black,citecolor=blue!50!black,urlcolor=blue!50!black]{hyperref}
\allowdisplaybreaks
\hypersetup{
  pdftitle={Generic polar divisors and flag residues for root-system zeta functions},
  pdfauthor={Jonas Matuzas},
  pdfsubject={Generic polar divisors, residue functions, flag residues, specialization, and local pole orders for Komori-Matsumoto-Tsumura zeta functions},
  pdfkeywords={zeta function of a root system, multivariable Witten zeta function, singular locus, residue, Mordell-Tornheim zeta function, Euler-Zagier zeta function}}

\newtheorem{theorem}{Theorem}[section]
\newtheorem{proposition}[theorem]{Proposition}
\newtheorem{corollary}[theorem]{Corollary}
\newtheorem{lemma}[theorem]{Lemma}
\newtheorem{definition}[theorem]{Definition}
\newtheorem{remark}[theorem]{Remark}
\newtheorem{example}[theorem]{Example}

\newcommand{\N}{\mathbb N}
\newcommand{\Nzero}{\mathbb N_0}
\newcommand{\C}{\mathbb C}
\newcommand{\Res}{\operatorname*{Res}}
\newcommand{\supp}{\operatorname{supp}}
\newcommand{\FP}{\operatorname*{FP}}
\newcommand{\dd}{\,\mathrm d}
\newcommand{\RF}[2]{(#1)_{#2}}
\newcommand{\cR}{\mathcal R}
\newcommand{\cP}{\mathcal P}
\newcommand{\cW}{\mathcal W}
\newcommand{\cA}{\mathcal A}
\newcommand{\bfone}{\mathbf 1}
\newcommand{\EZ}{\zeta_{\mathrm{EZ}}}
\newcommand{\FD}{F_D}
\newcommand{\Zero}{\mathcal Z}

\title{Generic polar divisors and flag residues\\
for root-system zeta functions}
\author{Jonas Matuzas\\
\href{mailto:jonas.matuzas@gmail.com}{\texttt{jonas.matuzas@gmail.com}}}
\date{}

\begin{document}
\maketitle

\begin{abstract}
Let \(\Phi\) be an irreducible crystallographic root system, and let
\(Z_\Phi(\mathbf s)\) denote the untwisted Komori--Matsumoto--Tsumura
zeta function with one exponent for each positive coroot.  For a
nonempty set \(S\) of simple nodes, write
\[
 H_{S,\ell}:\qquad
 \sum_{\alpha\in R(S)}s_\alpha=|S|-\ell .
\]
We prove that every proper-support hyperplane \(H_{S,\ell}\) is a
genuine polar divisor at a generic point, whereas exact homogeneity
leaves only the unshifted full-support divisor.  The residue on
\(H_{S,\ell}\) is expressed as a finite Taylor-jet sum of reduced
projective periods and polynomially weighted complementary
root-system zeta functions.

On the maximal support wonderful model, boundary terms are indexed by
strict decorated flags.  We derive recursive flag residues,
component-mass gamma factors, and an incidence-complete formula for
the Laurent coefficients on any transverse affine slice.  In
particular, the pole order is determined by the first nonzero
aggregate coefficient, not by the largest order of an individual
flag.

The general formulas recover the classical \(A_2\) and \(A_3\)
singular data, Zhao's Euler--Zagier residues, and the rank-two
\(C_2\) and \(G_2\) residue functions.  For \(B_3\) and \(C_3\) we
derive the carrier geometry and the lower-rank factorizations of the
positive residues, identify the three-term cancellation at
\(s=1/8\), and show that negative half-integers are the only possible
locations of double poles.  The known \(B_3\) double coefficient at
\(-1/2\) is recovered in the flag normalization.
\end{abstract}



\section{Introduction}

Komori, Matsumoto and Tsumura associate with a crystallographic root
system a multivariable Dirichlet series having one complex exponent
for each positive root \cite{KMT2010PLMS,KMTBook2023}.  We use the
equivalent dual convention in which the exponents are indexed by
positive coroots.  Its diagonal specialization is the normalized
single-variable Witten zeta function.  General continuation theorems
locate possible polar hyperplanes, but they do not by themselves
decide which candidates are genuine, what their residues are, or how
several incident divisors combine after specialization.

These questions are already nontrivial in low rank.  The
Mordell--Tornheim function gives the \(A_2\) case.  Matsumoto and
Tsumura determined the six-variable \(A_3\) singular locus
\cite{MatsumotoTsumura2006}, and Komori, Matsumoto and Tsumura treated
the relevant \(C_2\) and \(G_2\) families
\cite{KMT2010II,KMT2011G2}.  Zhao computed the residues of
Euler--Zagier multiple zeta functions in arbitrary depth
\cite{Zhao2000}, while Akiyama, Egami and Tanigawa determined their
surviving polar hyperplanes \cite{AET2001}.  Au obtained an independent
single-variable analysis in ranks two and three, including the
positive residues and the analysis of possible multiple poles in
types \(B_3\) and \(C_3\) \cite{Au2025}.

The first result of the present paper determines the generic polar
divisor in every irreducible crystallographic type.  If \(S\) is a
proper nonempty support and \(\ell\geq0\), then
\[
 \sum_{\alpha\in R(S)}s_\alpha=|S|-\ell
\]
is genuinely polar at a generic point.  Its residue is a finite sum
indexed by normal Taylor degrees; each term factors into a
projective period and a polynomially weighted zeta function of the
complementary subsystem.  For full support, exact homogeneity produces
only the unshifted radial hyperplane.

The second result describes intersections.  On the maximal support
wonderful model, the canonical boundary terms are indexed by strict
decorated flags.  Successive residues are encoded by relative face
operators.  When a relative layer is disconnected, integration over
the component masses introduces a quotient of gamma functions and may
lower the expected order.  On a transverse affine curve
\(\gamma(t)\), the coefficient of \(t^{-m}\) is the sum of the
appropriate derivatives of \emph{all} incident flag coefficients.
Consequently, different flags of the same candidate order may cancel;
the actual pole order is the largest \(m\) for which this aggregate is
nonzero.  We call this complete sum the aggregate Laurent coefficient.
A formula is incidence-complete if it includes every incident flag
contribution before that coefficient is tested for vanishing.

The proof uses a Schwinger--Mellin representation followed by a
wonderful-model resolution of the support subspaces.  The normalized
Mellin distributions are entire across ordinary coordinate faces, so
all parameter poles come from the Todd-denominator arrangement.  A
homogeneous interior-sector argument separates the unique
full-support radial pole.  At a proper face, Taylor expansion of the
crossing forms gives the explicit residue formula.  Iterating the same
construction along nested supports yields the flag calculus.

The general theory is then tested in three directions.  First, it
recovers the complete \(A_2\) and \(A_3\) residue structure, including
codimension-two and length-three flag coefficients.  Second, the
Euler--Zagier specialization reproduces Zhao's residue formulas and
explains the familiar thinning of candidate hyperplanes.  Third, the
\(B_3\) and \(C_3\) support ledgers lead to exact lower-rank
factorizations of the positive diagonal residues and to the
classification of possible double-pole locations.

The continuation and candidate-geometry background is due to
Lichtin, Essouabri, and subsequent work on multivariable Dirichlet
series \cite{Lichtin1988,Essouabri1997,BEL2007}.  Related approaches
include conical Laurent expansions \cite{GPZ2014,GPZ2020},
Newton-polytope descriptions of Shintani zeta functions
\cite{Lopez2022}, and explicit continuation algorithms
\cite{Rutard2023}.  Budur, Shi and Zuo recently determined the polar
loci and polar-order data of multivariable Archimedean local zeta
integrals attached to polynomial maps \cite{BudurShiZuo2025}.  Their
continuous local-integral setting is different from the discrete KMT
lattice series considered here, but it is a close comparison for the
geometric polar-locus problem.  The contribution here is the
combination of generic nonvanishing, explicit residue functions,
canonical flag coefficients, gamma defects, and aggregate
specialization formulas in the root-system setting.

Section~2 fixes the KMT convention and the carrier arrangement.
Section~3 proves the generic divisor and residue theorems.
Section~4 treats zero loci and specialization.  Section~5 develops
flag residues and local pole orders.  Sections~6--9 give the
\(A_2\), \(A_3\), Euler--Zagier, rank-two, and \(B_3/C_3\)
applications.  Section~10 discusses diagonal specialization.
\section{The KMT function and the carrier arrangement}

Let $\Psi=\Phi^\vee$ and choose positive coroots $\Psi^+$ and simple
coroots $\beta_1,\ldots,\beta_r$.  Write
\[
 \alpha=\sum_{i=1}^r b_i(\alpha)\beta_i,
 \qquad b_i(\alpha)\in\Nzero,
\]
and put
\[
 L_\alpha(\mathbf m)=\sum_{i=1}^r b_i(\alpha)m_i.
\]
We fix the untwisted, strongly dominant KMT variant
\begin{equation}\label{eq:KMT}
 Z_\Phi(\mathbf s)
 =\sum_{\mathbf m\in\N^r}
   \prod_{\alpha\in\Psi^+}L_\alpha(\mathbf m)^{-s_\alpha}.
\end{equation}
This is the $\mathbf y=0$ specialization of the standard KMT function.
If every $s_\alpha=s$, then \eqref{eq:KMT} is the normalized Witten
zeta function; the ordinary Witten function differs by the Weyl
normalizing factor raised to $s$.

For $\varnothing\ne S\subseteq I:=\{1,\ldots,r\}$ define
\[
 R(S)=\{\alpha\in\Psi^+:\supp(\alpha)\cap S\ne\varnothing\},
 \qquad N(S)=|R(S)|.
\]
For $\ell\in\Nzero$ put
\begin{equation}\label{eq:Lambda}
 \Lambda_{S,\ell}(\mathbf s)
 =\sum_{\alpha\in R(S)}s_\alpha-|S|+\ell,
 \qquad H_{S,\ell}=\{\Lambda_{S,\ell}=0\}.
\end{equation}
We call the affine functions just defined carrier forms and their zero
sets carrier hyperplanes.  The locally finite affine arrangement
\[
 \cA_\Phi=\bigcup_{\varnothing\ne S\subseteq I}
            \bigcup_{\ell\geq0}H_{S,\ell}
\]
is the decorated carrier arrangement.  Throughout,
\(\Res_{\Lambda=0}\) means the coefficient of \(\Lambda^{-1}\) when
\(\Lambda\) itself is used as the transverse coordinate; this convention
removes the usual scaling ambiguity in multivariable residues.

For later use split the positive coroots into
\begin{align*}
 \Psi^+_{S^c}&=\{\alpha:\supp(\alpha)\subseteq S^c\},\\
 C(S)&=\{\alpha:\supp(\alpha)\cap S\ne\varnothing,
                   \ \supp(\alpha)\cap S^c\ne\varnothing\}.
\end{align*}
Thus $R(S)$ consists of the roots supported in $S$ together with the
crossing roots $C(S)$.

\paragraph{Core notation.}
The following symbols are fixed throughout; this table also distinguishes
geometric supports from the later analytic coefficients.
\begin{center}
\small
\begin{tabularx}{\textwidth}{@{}l X@{}}
\toprule
symbol & meaning\\
\midrule
\(R(S),N(S)\) & positive coroots meeting \(S\), and their number\\
\(C(S)\) & coroots meeting both \(S\) and \(S^c\)\\
\(\Lambda_{S,\ell},H_{S,\ell}\) & carrier form and its affine hyperplane\\
\(\cP_{S,\mathbf k}\) & canonically reduced projective face period\\
\(\cW_{S,\mathbf k}\) & polynomially weighted complementary KMT series\\
\(C_{\mathfrak F}\) & aggregate holomorphic coefficient of a decorated flag\\
\(\cA_m(\gamma)\) & aggregate coefficient of order \(m\) on a transverse curve\\
\bottomrule
\end{tabularx}
\end{center}

The exact Schwinger representation, initially in an absolute-convergence
region, is
\begin{equation}\label{eq:Schwinger}
 Z_\Phi(\mathbf s)
 =\frac1{\prod_{\alpha}\Gamma(s_\alpha)}
 \int_{(0,\infty)^{|\Psi^+|}}
 \left(\prod_\alpha t_\alpha^{s_\alpha-1}\right)
 \prod_{i=1}^r\frac{1}{e^{A_i(\mathbf t)}-1}\,\dd\mathbf t,
\end{equation}
where $A_i(\mathbf t)=\sum_\alpha b_i(\alpha)t_\alpha$.

\begin{lemma}[Holomorphy across coordinate faces]
\label{lem:normalized-mellin}
For each \(\alpha\), the distribution
\[
 \frac{t_{\alpha,+}^{s_\alpha-1}}{\Gamma(s_\alpha)}
\]
extends entire in \(s_\alpha\).  Consequently, a coordinate boundary
\(t_\alpha=0\) at which every Todd denominator
\((e^{A_i(\mathbf t)}-1)^{-1}\) is smooth contributes no parameter pole.
All carrier poles in \eqref{eq:Schwinger} therefore come from support
subspaces on which one or more of the linear forms \(A_i\) vanish.
\end{lemma}

\begin{proof}
For \(\Re s>0\), pair the displayed distribution with a compactly
supported smooth test function.  For every integer \(m\geq0\),
\[
 \frac{t_+^{s-1}}{\Gamma(s)}
 =\partial_t^m\!\left(\frac{t_+^{s+m-1}}{\Gamma(s+m)}\right)
\]
as distributions.  Choosing \(m\) so that \(\Re(s+m)>0\) gives an
entire continuation, and the identity is compatible on overlaps.  Tensor
products give the corresponding statement for several coordinate
variables.  Multiplication by a smooth Todd factor preserves
holomorphy; at Schwinger infinity the exponential decay used in
\eqref{eq:Schwinger} is uniform on compact parameter sets.
\end{proof}

\begin{lemma}[Support-subspace lattice]
\label{lem:support-lattice}
On the closed positive Schwinger orthant, put
\[
 E_S=\{\mathbf t:t_\alpha=0\text{ for every }\alpha\in R(S)\}.
\]
Then
\[
 E_S=\{\mathbf t:A_i(\mathbf t)=0\text{ for every }i\in S\},
 \qquad E_S\cap E_T=E_{S\cup T}.
\]
The map \(S\mapsto E_S\) is injective.  Hence the boundary support
lattice is the Boolean lattice of nonempty simple-node supports.
\end{lemma}

\begin{proof}
All coefficients \(b_i(\alpha)\) and all Schwinger variables are
nonnegative.  Thus \(A_i(\mathbf t)=0\) holds exactly when
\(t_\alpha=0\) for every root with \(b_i(\alpha)>0\).  Imposing this for
all \(i\in S\) gives precisely \(R(S)\).  The intersection identity
follows from \(R(S)\cup R(T)=R(S\cup T)\).  Finally, the simple coroot
\(\beta_i\) belongs to \(R(S)\) exactly when \(i\in S\), so distinct
supports give distinct subspaces.
\end{proof}

After Lemma~\ref{lem:normalized-mellin}, only the support subspaces on
which Todd denominators vanish need to be resolved.  Their resolution
gives the same divisors as the facewise calculation below.  Wonderful
models make these divisors normal crossing
\cite{DeConciniProcesi1995,Li2009}.  Bare meromorphic continuation also
follows from the established KMT and general polynomial-Dirichlet
theories \cite{Essouabri1997,KMT2010PLMS}.

\begin{lemma}[Orders in a maximal-support chart]
\label{lem:chart-orders}
Let
\[
 S_1\subsetneq\cdots\subsetneq S_k
\]
be a strict support flag, and let
\(D_{S_1},\ldots,D_{S_k}\) be the corresponding boundary divisors in
the maximal wonderful model of the support-subspace arrangement.
After localizing away from ordinary coordinate faces, near the relative
interior of their intersection there are boundary coordinates \(\rho_1,\ldots,\rho_k\), remaining coordinates
\(\mathbf u\), and positive smooth units such that
\begin{align}
 \pi^*t_\alpha
 &=u_\alpha(\boldsymbol\rho,\mathbf u)
   \prod_{j:\,\alpha\in R(S_j)}\rho_j,
 \label{eq:chart-t}\\
 \left|\pi^*(\dd\mathbf t)\right|
 &=u_J(\boldsymbol\rho,\mathbf u)
   \prod_{j=1}^k\rho_j^{N(S_j)-1}
   \dd\boldsymbol\rho\,\dd\mathbf u,
 \label{eq:chart-jac}\\
 \pi^*A_i
 &=v_i(\boldsymbol\rho,\mathbf u)
   \prod_{j:\,i\in S_j}\rho_j.
 \label{eq:chart-A}
\end{align}
All displayed units are nonzero on the relative interior.  Therefore,
in an initial absolute-convergence region, the pullback of the
Schwinger density has the form
\begin{equation}
 \left(\prod_{j=1}^k
 \rho_j^{\Lambda_{S_j,0}(\mathbf s)-1}\right)
 a_{\mathfrak F}(\boldsymbol\rho,\mathbf u,\mathbf s)
 \dd\boldsymbol\rho\,\dd\mathbf u,
 \label{eq:chart-density}
\end{equation}
where \(a_{\mathfrak F}\) is smooth in the resolved spatial variables
and holomorphic in \(\mathbf s\).  A Taylor monomial
\(\rho_1^{\ell_1}\cdots\rho_k^{\ell_k}\) replaces
\(\Lambda_{S_j,0}\) by \(\Lambda_{S_j,\ell_j}\).
\end{lemma}

\begin{proof}
For the coordinate-subspace arrangement, the divisorial valuation of
\(t_\alpha\) along \(D_S\) is one precisely when
\(\alpha\in R(S)\), and is zero otherwise.  This gives
\eqref{eq:chart-t}.  Since \(E_S\) has codimension \(N(S)\), the
Jacobian discrepancy of its blow-up is \(N(S)-1\); iteration along a
nested flag gives \eqref{eq:chart-jac}.  Finally, all coefficients in
\(A_i=\sum_\alpha b_i(\alpha)t_\alpha\) are nonnegative.  Along
\(D_S\), every summand has positive order if \(i\in S\), and the
simple-coroot term \(t_{\beta_i}\) has order exactly one.  If
\(i\notin S\), that same term has order zero.  Positivity prevents
cancellation, proving \eqref{eq:chart-A}.

Write
\[
 \frac1{e^z-1}=z^{-1}\frac{z}{e^z-1}.
\]
The second factor is analytic at zero.  Hence the exponent of
\(\rho_j\) in the pullback of \eqref{eq:Schwinger} is
\[
 \sum_{\alpha\in R(S_j)}(s_\alpha-1)
 +(N(S_j)-1)-|S_j|
 =\Lambda_{S_j,0}(\mathbf s)-1.
\]
The normalized gamma factors are entire, and all remaining factors
are absorbed into the smooth holomorphic amplitude.  Taylor monomials
supply the stated nonnegative integer shifts.
\end{proof}

\begin{lemma}[Uniform normal-crossing Mellin continuation]
\label{lem:uniform-mellin}
Let \(K\subset\C^d\) be compact, let
\(\lambda_1,\ldots,\lambda_k\) be affine-linear functions, and let
\(a(\boldsymbol\rho,\mathbf u,\mathbf s)\) be smooth for
\(\boldsymbol\rho\in[0,1]^k\) and holomorphic in a neighborhood of
\(K\).  Assume either that it is compactly supported in
\(\mathbf u\), or that every mixed spatial and parameter derivative
used below is uniformly \(L^1\) in \(\mathbf u\) on compact parameter
sets.  Initially where
all \(\Re\lambda_j>0\), put
\[
 I(\mathbf s)=\int_{[0,1]^k\times U}
 \left(\prod_{j=1}^k\rho_j^{\lambda_j(\mathbf s)-1}\right)
 a(\boldsymbol\rho,\mathbf u,\mathbf s)
 \dd\boldsymbol\rho\,\dd\mathbf u.
\]
Then, in a neighborhood of \(K\), \(I\) is a finite sum of terms
\[
 h_{J,\mathbf n}(\mathbf s)
 \prod_{j\in J}\frac1{\lambda_j(\mathbf s)+n_j}
\]
plus a holomorphic function, where
\(J\subseteq\{1,\ldots,k\}\), \(n_j\in\Nzero\), and all
\(h_{J,\mathbf n}\) are holomorphic.  Only finitely many \(n_j\)
occur on a fixed compact parameter set.  Every parameter derivative of
\(I\) admits an expansion of the same shape in which the denominators are
finite products
\(\prod_{j\in J}(\lambda_j(\mathbf s)+n_j)^{-m_j}\) with integers
\(m_j\ge1\), again with holomorphic numerator functions; on a fixed
compact parameter set only finitely many exponents occur.  Differentiation
may thus raise the order of the explicit affine denominators but introduces
no new polar loci.
\end{lemma}

\begin{proof}
Use successively the one-variable identity
\begin{equation}
 \int_0^1\rho^{\lambda-1}f(\rho)\,\dd\rho
 =\sum_{n=0}^{M-1}\frac{f^{(n)}(0)}{n!(\lambda+n)}
  +\int_0^1\rho^{\lambda+M-1}f_M(\rho)\,\dd\rho,
 \label{eq:one-variable-mellin}
\end{equation}
where Taylor's formula writes
\(f(\rho)=\sum_{n<M}f^{(n)}(0)\rho^n/n!+\rho^Mf_M(\rho)\).
Choose \(M\) so that
\(\Re(\lambda(\mathbf s)+M)>0\) throughout the parameter
neighborhood under consideration.  The remainder is then holomorphic; in the noncompact case this uses
the assumed uniform \(L^1\)-bounds in \(\mathbf u\).
Parameter derivatives insert powers of \(\log\rho\); for every
\(\varepsilon>0\) and integer \(q\geq0\),
\[
 \int_0^1\rho^{\varepsilon-1}|\log\rho|^q\,\dd\rho
 =\frac{q!}{\varepsilon^{q+1}},
\]
so differentiation under the integral is justified uniformly on
compact sets.  Iterating \eqref{eq:one-variable-mellin} in the
\(k\) boundary variables proves the assertion.  The same argument,
with the remaining smooth variables treated as parameters, proves
holomorphy of every coefficient.
\end{proof}

\begin{lemma}[Generic flag residues and uniqueness of the aggregate coefficients]
\label{lem:flag-canonicity}
For a strict decorated flag
\[
 \mathfrak F=((S_1,\ell_1),\ldots,(S_k,\ell_k)),
 \qquad S_1\subsetneq\cdots\subsetneq S_k,
\]
the affine forms
\(\Lambda_{S_1,\ell_1},\ldots,\Lambda_{S_k,\ell_k}\) have
linearly independent differentials.  At a point of their intersection
lying on no additional carrier hyperplane, the coefficient of
\(\prod_{j=1}^k\Lambda_{S_j,\ell_j}^{-1}\) is therefore an intrinsic
multivariable residue.  On a higher carrier intersection, individual
resolved terms may be redistributed by refinement, but their complete
Laurent coefficient on any transverse affine slice is independent of
all resolution choices.
\end{lemma}

\begin{proof}
Choose \(i_j\in S_j\setminus S_{j-1}\), with \(S_0=\varnothing\),
and evaluate the coefficient vectors of the differentials on the
simple coroots \(\beta_{i_j}\).  Since
\(\beta_{i_j}\in R(S_m)\) exactly when \(m\geq j\), the resulting
matrix is triangular with diagonal entries one.  The differentials are
therefore independent and may be completed to affine coordinates near
a generic point of the flag intersection.  There the coefficient is
the ordinary multivariable residue and is unique.

For a generic flag, the coefficient may equivalently be taken as the
iterated normal residue of the globally pulled-back density along
\(D_{S_1}\cap\cdots\cap D_{S_k}\).  Normal residues are local and
additive, so a partition of unity summing to one gives the same global
coefficient, and coordinate changes preserve it.  At a point where
further carriers meet, pull the meromorphic germ back to a transverse
affine slice.  Every resolved construction gives the same one-variable
meromorphic germ; hence every Laurent coefficient of that pullback is
unique.  Summing all resolved terms contributing to a given Laurent
order gives the same result for every chart, partition of unity,
blow-up coordinate system, and Taylor truncation.  This is the uniqueness of the aggregate coefficients used below.
\end{proof}

\begin{lemma}[Uniform Schwinger-tail estimate]
\label{lem:schwinger-tail}
On every resolved support chart, factor the exceptional boundary powers
as in \eqref{eq:chart-density}.  On the part of that chart mapping to
\(\lVert\mathbf t\rVert_1\ge1\), the remaining coefficient and all its
spatial and parameter derivatives are uniformly integrable in the
unbounded tangential variables on compact parameter sets.  Consequently
the tail has a normal-crossing Mellin expansion with the same carrier
forms \(\Lambda_{S,\ell}\); Schwinger infinity introduces no additional
affine polar form.
\end{lemma}

\begin{proof}
For fixed rank \(r\), there is a constant \(C_r\) such that, for every
\(x>0\),
\[
 \frac1{e^x-1}\le C_r e^{-x/(2r)}(1+x^{-1}).
\]
Moreover,
\[
 \sum_{i=1}^r A_i(\mathbf t)
 =\sum_{\alpha\in\Psi^+}\operatorname{ht}(\alpha)t_\alpha
 \ge \lVert\mathbf t\rVert_1.
\]
Thus the Todd product is bounded by
\[
 C e^{-c\lVert\mathbf t\rVert_1}
 \prod_{i=1}^r(1+A_i(\mathbf t)^{-1}).
\]
On a resolved chart, every singular factor \(A_i^{-1}\) is the
exceptional monomial already extracted in \eqref{eq:chart-density},
times the inverse of a positive smooth unit.  Residual coordinate faces
are entire by Lemma~\ref{lem:normalized-mellin}.  After this extraction,
powers of the Schwinger variables and all parameter derivatives
contribute only polynomial and logarithmic growth in the remaining
variables.  Exponential decay therefore gives uniform \(L^1\)-bounds
for the coefficient and all required derivatives.  The noncompact form
of Lemma~\ref{lem:uniform-mellin} applies and produces only the same
exceptional Mellin denominators.  Hence infinity changes the
holomorphic coefficients but creates no new carrier form.
\end{proof}

\begin{theorem}[Local flag normal form]\label{thm:flag-normal-form}
The function \(Z_\Phi(\mathbf s)\) continues meromorphically to
\(\C^{|\Psi^+|}\).  Its polar divisor is supported on \(\cA_\Phi\).
Locally one has
\begin{equation}\label{eq:flag-normal-form}
 Z_\Phi(\mathbf s)=H(\mathbf s)+
 \sum_{\mathfrak F}C_{\mathfrak F}(\mathbf s)
 \prod_{(S,\ell)\in\mathfrak F}
 \Lambda_{S,\ell}(\mathbf s)^{-1},
\end{equation}
where \(H\) and the aggregate coefficients \(C_{\mathfrak F}\) are
holomorphic.  On the maximal support wonderful model,
\(\mathfrak F\) ranges over strict decorated support flags, and the sum
is finite in every bounded parameter neighborhood.  At a generic flag intersection the corresponding coefficient is the
intrinsic multivariable residue of Lemma~\ref{lem:flag-canonicity}; at
higher intersections only the complete aggregate Laurent coefficients
are asserted to be intrinsic.
\end{theorem}

\begin{proof}
Work on a compact parameter polydisc and insert a smooth spatial
partition of unity in \eqref{eq:Schwinger}, including a compact/tail
radial partition.  Coordinate faces on which the Todd factors are
smooth are holomorphic by Lemma~\ref{lem:normalized-mellin}; hence it
remains to resolve the support-subspace arrangement.

Use the maximal wonderful model.  By
Lemma~\ref{lem:support-lattice} and the Feichtner--Kozlov nested-set
criterion, its boundary intersections are precisely strict support
chains.  On each chart, Lemma~\ref{lem:chart-orders} reduces every
localized contribution to a finite sum of normal-crossing Mellin
integrals with boundary forms \(\Lambda_{S,0}\).  In the compact
part the amplitudes have compact tangential support; in the tail they
satisfy the uniform \(L^1\)-bounds of
Lemma~\ref{lem:schwinger-tail}.  Apply
Lemma~\ref{lem:uniform-mellin}.  Every polar denominator is
\(\Lambda_{S,0}+\ell=\Lambda_{S,\ell}\), with
\(\ell\in\Nzero\), and each boundary variable occurs to at most the
first power in a denominator.  Products of denominators occur only for
the strict flags indexing the chart.

The resolution is global before any projective coefficient is
integrated.  Thus a boundary pole of a would-be projective link is
already represented by an additional divisor of the same wonderful
model, hence by a longer strict flag; no new affine polar form is
created by a later continuation step.  The uniform Mellin lemma makes
the coefficient left after all incident boundary variables have been
expanded holomorphic on the chosen parameter neighborhood.  Only
finitely many shifts can meet that neighborhood, so the local sum is
finite.  Lemma~\ref{lem:flag-canonicity} identifies the coefficient at a
generic flag intersection with its intrinsic multivariable residue and
shows that every aggregate coefficient obtained after specialization is
independent of the resolved presentation.

The argument applies on every compact parameter set, with a sufficiently
high but finite Taylor order chosen there.  Lemma~\ref{lem:schwinger-tail}
shows that the noncompact pieces contribute the same carrier forms and
normally convergent holomorphic coefficients.  These local
continuations agree on overlaps by uniqueness, giving meromorphic
continuation on the whole parameter space and the normal form
\eqref{eq:flag-normal-form}.
\end{proof}

\section{Generic divisors and explicit residue functions}

We first compute the residues on proper supports and then treat the
full-support contribution separately.

For $S\subsetneq I$, $\alpha\in R(S)$ and
$\mathbf u\in(0,\infty)^S$, $\mathbf n\in\N^{S^c}$, write
\[
 A_{\alpha,S}(\mathbf u)=\sum_{i\in S}b_i(\alpha)u_i,
 \qquad
 B_{\alpha,S}(\mathbf n)=\sum_{j\in S^c}b_j(\alpha)n_j.
\]
For $\alpha\in C(S)$ both quantities are positive.  Let
\[
 \Sigma_S=\{\mathbf u\in(0,\infty)^S:\sum_{i\in S}u_i=1\}.
\]
On \(\Sigma_S\), \(\dd\mathbf u\) denotes the standard Dirichlet
simplex measure: after eliminating any one coordinate, it is Lebesgue
measure in the remaining \(|S|-1\) coordinates.  The transition between
two such charts has determinant \(\pm1\), and a singleton simplex has
mass one.  This is the measure arising from the radial change of
variables \(x_i=R u_i\).
If $\mathbf k=(k_\alpha)_{\alpha\in C(S)}\in\Nzero^{C(S)}$, extend it
by $k_\alpha=0$ on $R(S)\setminus C(S)$ and put
\begin{align}
 \cP_{S,\mathbf k}(\mathbf s)
 &:=\FP\int_{\Sigma_S}
      \prod_{\alpha\in R(S)}
      A_{\alpha,S}(\mathbf u)^{-s_\alpha-k_\alpha}\,\dd\mathbf u,
      \label{eq:Psk}\\
 \cW_{S,\mathbf k}(\mathbf s)
 &:=\sum_{\mathbf n\in\N^{S^c}}
      \prod_{\beta\in\Psi^+_{S^c}}
      B_{\beta,S}(\mathbf n)^{-s_\beta}
      \prod_{\alpha\in C(S)}
      B_{\alpha,S}(\mathbf n)^{k_\alpha}.
      \label{eq:Wsk}
\end{align}
The finite part in \eqref{eq:Psk} means the canonical meromorphic value;
in an ordinary convergence chamber it is the displayed integral.
Likewise \eqref{eq:Wsk} is continued from its convergence chamber.  We
write
\[
 \cR_{S,\ell}(\mathbf s)
 :=\Res_{\Lambda_{S,\ell}=0}Z_\Phi(\mathbf s)
\]
for the resulting residue function on \(H_{S,\ell}\).

\begin{lemma}[Homogeneous interior sector and full-support shifts]
\label{lem:interior-sector}
Let \(\chi\in C_c^\infty(\Sigma_I^\circ)\), let
\(\eta\in C^\infty([0,\infty))\) vanish near zero and equal one for
large arguments, and write \(|x|_1=\sum_i x_i\).  Put
\[
 \sigma(\mathbf s)=\sum_{\alpha\in\Psi^+}s_\alpha
\]
and
\[
 Z_\chi(\mathbf s)=
 \sum_{\mathbf m\in\N^r}
 \eta(|\mathbf m|_1)\,
 \chi\!\left(\frac{\mathbf m}{|\mathbf m|_1}\right)
 \prod_{\alpha\in\Psi^+}L_\alpha(\mathbf m)^{-s_\alpha}.
\]
Then \(Z_\chi\) is holomorphic in the full-support radial variable
except for a possible simple pole at \(\sigma(\mathbf s)=r\).  Its
residue there is
\[
 \int_{\Sigma_I}
 \chi(\mathbf u)
 \prod_{\alpha\in\Psi^+}L_\alpha(\mathbf u)^{-s_\alpha}
 \,\dd\mathbf u.
\]
In particular, an interior sector produces no full-support shifted
hyperplane \(H_{I,\ell}\) with \(\ell\geq1\).  The complementary angular
pieces are localized near proper coordinate faces and belong to proper
supports in the local flag expansion.
\end{lemma}

\begin{proof}
Choose a smooth dyadic partition of unity in the radial variable.  On
each annulus the amplitude
\[
 \chi\!\left(\frac{x}{|x|_1}\right)
 \prod_\alpha L_\alpha(x)^{-s_\alpha}
\]
is smooth, is supported a positive distance from every coordinate and
root hyperplane, and is exactly homogeneous of degree
\(-\sigma(\mathbf s)\).  Extend it by zero to \(\mathbb R^r\) and apply
Poisson summation annulus by annulus.  The zero Fourier mode is the
ordinary cone integral.  After summing the dyadic partition, its polar
part is
\[
 \left(\int_{\Sigma_I}\chi(\mathbf u)
       \prod_\alpha L_\alpha(\mathbf u)^{-s_\alpha}
       \,\dd\mathbf u\right)
 \int_1^\infty R^{r-1-\sigma(\mathbf s)}\,\dd R,
\]
which has the single denominator \(\sigma(\mathbf s)-r\).

To make the nonzero-mode assertion uniform, let \(a_{j,\mathbf s}\)
be the amplitude on the dyadic annulus \(|x|_1\asymp2^j\).  After the
change \(x=2^jy\), repeated integration by parts gives, for every compact
parameter set \(K\) and every \(N\geq0\),
\[
 \bigl|\widehat a_{j,\mathbf s}(k)\bigr|
 \le C_{K,N}\,
 2^{j(r-\Re\sigma(\mathbf s))}
 (1+2^j|k|)^{-N},
 \qquad \mathbf s\in K,\quad k\in\mathbb Z^r\setminus\{0\}.
\]
The same estimate holds after any parameter derivative, with an
additional polynomial factor in \(j\), which is absorbed by increasing
\(N\).  Choosing \(N\) larger than \(r+\sup_K|r-\Re\sigma|\) makes the
double sum over \(j\) and \(k\neq0\) locally normally convergent.
Consequently the sum of all nonzero Fourier modes is entire in the
full exponent vector, not merely in \(\sigma\).  The finite-radius
cutoff also contributes an entire term.
Finally, a smooth partition of unity on the closed simplex separates the
interior cutoffs considered above from neighborhoods of its proper
faces.  Those face neighborhoods are exactly the lower-dimensional
support sectors resolved in Theorem~\ref{thm:flag-normal-form}.  Hence no
full-support shifted divisor remains.
\end{proof}

\begin{theorem}[Residue formula on a proper support]\label{thm:residue}
Let $\varnothing\ne S\subsetneq I$ and $\ell\geq0$.  At a generic point
of $H_{S,\ell}$,
\begin{equation}\label{eq:residue-general}
 \Res_{\Lambda_{S,\ell}=0} Z_\Phi(\mathbf s)
 =(-1)^\ell
 \sum_{\substack{\mathbf k\in\Nzero^{C(S)}\\|\mathbf k|=\ell}}
 \left(\prod_{\alpha\in C(S)}
       \frac{\RF{s_\alpha}{k_\alpha}}{k_\alpha!}\right)
 \cP_{S,\mathbf k}(\mathbf s)\,
 \cW_{S,\mathbf k}(\mathbf s).
\end{equation}
For $\ell=0$ this reduces to
\begin{equation}\label{eq:unshifted}
 \Res_{H_{S,0}}Z_\Phi
 =\cP_{S,\mathbf0}(\mathbf s)\,
   Z_{\Phi_{S^c}}(\mathbf s|_{\Psi^+_{S^c}}).
\end{equation}
For $S=I$, the only generic polar hyperplane is
\begin{equation}\label{eq:full}
 H_{I,0}:\quad \sum_{\alpha\in\Psi^+}s_\alpha=r,
\end{equation}
and its residue is
\begin{equation}\label{eq:fullres}
 \Res_{H_{I,0}}Z_\Phi
 =\FP\int_{\Sigma_I}
   \prod_{\alpha\in\Psi^+}
   L_\alpha(\mathbf u)^{-s_\alpha}\,\dd\mathbf u.
\end{equation}
No $H_{I,\ell}$ with $\ell\geq1$ is a component of the polar divisor.
\end{theorem}

\begin{proof}
For every crossing root use the Mellin--Barnes identity
\[
 (A+B)^{-s}=\frac1{2\pi i}
 \int_{(c)}\frac{\Gamma(s+z)\Gamma(-z)}{\Gamma(s)}
 A^{-s-z}B^z\,\dd z,
 \qquad -\Re s<c<0.
\]
Shift the contours to the right.  The pole of $\Gamma(-z)$ at
$z=k\in\Nzero$ contributes
$(-1)^k\RF{s}{k}/k!$.  A multi-residue $\mathbf k$ therefore separates
an $S$-block from a polynomially weighted $S^c$-block, giving
\eqref{eq:Psk} and \eqref{eq:Wsk}.  The $S$-block is homogeneous of total
degree
\[
 \sum_{\alpha\in R(S)}s_\alpha+|\mathbf k|.
\]
Its interior radial integral has a simple pole when this total equals
$|S|$.  Proper boundary faces of the $S$-block are assigned to proper
subsupports by Theorem~\ref{thm:flag-normal-form}, so away from their hyperplanes
the residue is exactly the simplex period.  Summing the multi-residues
with $|\mathbf k|=\ell$ proves \eqref{eq:residue-general}.  The
Stirling estimate gives exponential decay of
$\Gamma(s_\alpha+z)\Gamma(-z)$ on every shifted vertical contour,
uniformly for $\mathbf s$ in compact sets avoiding the crossed poles.
Together with absolute convergence in the initial chamber and the
recursive continuation of the two separated blocks, this justifies the
finite contour shifts and the parameterwise residue extraction.

If $S=I$, there is no external block and no crossing root.  Apply
Lemma~\ref{lem:interior-sector} to a smooth projective partition of
unity.  Its interior terms have the single radial denominator
$\sum_\alpha s_\alpha-r$ and residue \eqref{eq:fullres}; every
remaining angular piece is assigned to a proper support by
Theorem~\ref{thm:flag-normal-form}.  Thus no generic $H_{I,\ell}$ with
$\ell\geq1$ survives.  This proves \eqref{eq:full}--\eqref{eq:fullres}
and the last assertion.
\end{proof}

\begin{theorem}[True generic polar divisor]\label{thm:true-locus}
For irreducible $\Phi$, the support of the polar divisor is exactly
\begin{equation}\label{eq:true-divisor}
 H_{I,0}\ \cup\!
 \bigcup_{\varnothing\ne S\subsetneq I}
 \ \bigcup_{\ell\geq0}H_{S,\ell}.
\end{equation}
Every displayed hyperplane is generically a simple pole.
\end{theorem}

\begin{proof}
Containment follows from Theorem~\ref{thm:flag-normal-form}.  We give an explicit
nonvanishing point for every proper support, while keeping every factor
in the finite residue sum holomorphic.  Let \(d=|S|\), let
\(\gamma\) be the highest positive coroot, and let
\[
 \mathscr K_\ell=\{\mathbf k\in\mathbb N_0^{C(S)}:|\mathbf k|=\ell\}.
\]
Irreducibility implies that all simple-coroot coefficients of \(\gamma\)
are strictly positive, so \(\gamma\in C(S)\) and
\[
 0<c_-\le A_{\gamma,S}(\mathbf u)\le c_+<\infty
 \qquad(\mathbf u\in\overline{\Sigma_S}).
\]

Put \(t=\ell+\tfrac12\).  Choose positive numbers \(\eta_i\), \(i\in S\),
with \(\sum_{i\in S}\eta_i=t\), subject to the following finite genericity
condition.  For every \(\mathbf k\in\mathscr K_\ell\), avoid all
proper-subsupport pole hyperplanes of the reduced projective factor
\(\mathcal P_{S,\mathbf k}\).  Such a pole arising from
\(T\subsetneq S\) imposes an affine equation of the form
\[
 \sum_{i\in T}\eta_i=c_{T,\mathbf k,m},
\]
and cannot contain the whole positive simplex
\(\sum_{i\in S}\eta_i=t\).  Only finitely many integers \(m\) meet its
compact closure at the fixed level \(\ell\), so the required choice
exists.

Set
\[
 s_\gamma=\frac12,\qquad
 s_{\beta_i}=1-\eta_i\quad(i\in S),
\]
and set every other exponent in \(R(S)\) equal to zero.  Then
\[
 \sum_{\alpha\in R(S)}s_\alpha
 =d-t+\frac12=d-\ell,
\]
so the point lies on \(H_{S,\ell}\).  On the complementary subsystem
choose the simple-coroot exponents to be sufficiently large positive
generic real numbers.  Since every polynomial weight occurring for
\(\mathbf k\in\mathscr K_\ell\) has degree at most \(\ell\), one common
choice places every \(\mathcal W_{S,\mathbf k}\) in its absolute
convergence region.  Thus all factors in the finite sum
\eqref{eq:residue-general} are holomorphic at the chosen point; in
particular, no product is interpreted as \(0\cdot\infty\).

Every summand except \(k_\gamma=\ell\) now contains a factor
\((0)_{k_\alpha}\) with \(k_\alpha>0\), and hence vanishes.  The surviving
projective integrand is a bounded positive multiple of
\[
 \prod_{i\in S}u_i^{-1+\eta_i},
\]
whose Dirichlet integral converges because every \(\eta_i>0\).  Its
complementary series is absolutely convergent and strictly positive, and
\((\tfrac12)_\ell>0\).  Therefore the selected summand, and hence the
whole residue function, is finite and nonzero.  A small perturbation
inside \(H_{S,\ell}\) avoids all other carrier hyperplanes while
preserving this nonvanishing.  Thus the residue germ is not identically
zero.

For full support, take the diagonal point $s_\alpha=2/h$.  Since
$|\Psi^+|=rh/2$, it lies on $H_{I,0}$.  If
$\varnothing\ne T\subsetneq I$ and the complementary parabolic system
has irreducible components of ranks $r_a$ and Coxeter numbers $h_a<h$,
then
\[
 \frac{2}{h}|R(T)|-|T|
 =\sum_a r_a\left(1-\frac{h_a}{h}\right)>0.
\]
The resolved-face integrability criterion (the same strict
proper-parabolic estimate used in \cite{MatuzasLeading2026}) therefore
makes \eqref{eq:fullres} an ordinary finite positive simplex integral at this
point.  Thus $H_{I,0}$ is also genuinely polar.  Generic simplicity now
follows because a generic point of one hyperplane lies on no other
component of the locally finite arrangement.
\end{proof}

\begin{remark}
Theorem~\ref{thm:true-locus} concerns divisors at generic points.  At
intersections, compatible flag coefficients can have additional
zeros, including disconnected reciprocal-gamma zeros.  Thus the theorem
does not assert that the pole order at an intersection equals the number
of incident hyperplanes.
\end{remark}

\section{Residue zero loci and specialization thinning}

We call the disappearance of candidate polar components under a
specialization map specialization thinning.

The crossing-root Taylor product may be written in unequal-weight Bell
form.  Put
\[
 P_j=\sum_{\alpha\in C(S)}s_\alpha z_\alpha^j.
\]
Then
\[
 \prod_{\alpha\in C(S)}(1+t z_\alpha)^{-s_\alpha}
 =\exp\left(\sum_{j\ge1}\frac{(-1)^j}{j}P_jt^j\right)
 =\sum_{\ell\ge0}\mathbf B_\ell t^\ell,
\]
with
\begin{align*}
 \mathbf B_0&=1,&
 \mathbf B_1&=-P_1,\\
 \mathbf B_2&=\frac{P_1^2+P_2}{2},&
 \mathbf B_3&=-\frac{P_1^3+3P_1P_2+2P_3}{6},\\
 \mathbf B_4&=
 \frac{P_1^4+6P_1^2P_2+3P_2^2+8P_1P_3+6P_4}{24}.
\end{align*}
These polynomials are generically nonzero in every degree.  Hence there
is no generic off-diagonal parity thinning.

For a proper support, write the summands of
\eqref{eq:residue-general} as
\[
 T_{S,\mathbf k}(\mathbf s)=
 \left(\prod_{\alpha\in C(S)}
       \frac{(s_\alpha)_{k_\alpha}}{k_\alpha!}\right)
 \cP_{S,\mathbf k}(\mathbf s)\cW_{S,\mathbf k}(\mathbf s).
\]
On the generic part of \(H_{S,\ell}\), define the forced-zero set
\[
 \Zero^{\rm forced}_{S,\ell}
 =
 \bigcap_{|\mathbf k|=\ell}
 \left(
   \Zero^{\rm Poch}_{S,\mathbf k}
   \cup\Zero^{\rm per}_{S,\mathbf k}
   \cup\Zero^{\rm comp}_{S,\mathbf k}
 \right),
\]
where the three terms denote the zero sets of the Pochhammer product,
the canonical projective period, and the complementary weighted zeta
function.  Define the residual cancellation set by
\[
 \Zero^{\rm cancel}_{S,\ell}
 =
 \left\{\mathbf s:
 \sum_{|\mathbf k|=\ell}T_{S,\mathbf k}(\mathbf s)=0,\quad
 \text{not every }T_{S,\mathbf k}(\mathbf s)=0\right\}.
\]

\begin{proposition}[Zero locus of a generic residue]\label{prop:zero-locus}
Away from poles of the individual factors, the set-theoretic zero locus on
\(H_{S,\ell}\) is exactly
\[
 \Zero\!\left(
 \Res_{H_{S,\ell}}Z_\Phi\right)
 =
 \Zero^{\rm forced}_{S,\ell}
 \cup\Zero^{\rm cancel}_{S,\ell}.
\]
The Pochhammer part is an explicit union of affine hyperplanes.
The projective and complementary factors give a structural
factorization; their complete zero loci, and the associated
multiplicities, need not be known in general.
\end{proposition}

\begin{proof}
This is the exact finite sum \eqref{eq:residue-general}.  If every
summand has a zero, the residue vanishes for a forced factor reason.
Otherwise its vanishing is precisely cancellation among the nonzero
summands.  No further mechanism exists.
\end{proof}

\begin{proposition}[Sharp specialization criterion]\label{prop:thin}
Let \(\iota:\C^d\to\C^{|\Psi^+|}\) be affine and let \(K=\{L=0\}\)
be a source hyperplane.  Suppose the carrier forms associated with a compatible collection of
flags over a generic point of \(K\) pull back as
\[
 \iota^*\Lambda_v=c_vL,\qquad c_v\ne0.
\]
Then
\begin{equation}\label{eq:pullback-general}
 [L^{-m}]\,\iota^*Z_\Phi
 =
 \sum_{|\mathfrak F|\ge m}
 \frac{
   \left.
   \partial_L^{|\mathfrak F|-m}
   \iota^*C_{\mathfrak F}
   \right|_{L=0}
 }{(|\mathfrak F|-m)!\prod_{v\in\mathfrak F}c_v}.
\end{equation}
If only one generic divisor maps to \(K\), that divisor is lost exactly
when \(\iota(K)\) lies in its residue zero locus.  If several
compatible divisors coalesce, \(K\) is lost exactly when every
incidence-complete coefficient in \eqref{eq:pullback-general} vanishes.
\end{proposition}

\begin{proof}
Pull back the local flag expansion \eqref{eq:flag-normal-form}.  A flag of length \(k\) contributes
\(L^{-k}(\iota^*C_{\mathfrak F})/\prod_vc_v\).
Taylor expansion of the holomorphic coefficient gives
\eqref{eq:pullback-general}.  The final two statements are the
one-divisor and collapsing-divisor specializations of this identity.
\end{proof}

The collapsing clause is nonvacuous: in the depth-three
Euler--Zagier pullback below, two \(A_3\) divisor families coalesce onto
the last-two-variable lines.  The incidence sum reduces to a Bernoulli
coefficient and removes exactly the parity-gap lines.

\section{Flag residues and local pole orders}
\label{sec:flags}

The generic divisor theorem does not determine the local Laurent
structure where several carrier hyperplanes meet.  We now record the
canonical normal form supplied by the maximal support wonderful model
and the correct incidence-complete order law.

\subsection{Nested supports and the maximal model}

\begin{definition}[Incident and contributing flags]
A \emph{decorated support flag} is a strict chain
\[
 \mathfrak F=((S_1,\ell_1),\ldots,(S_k,\ell_k)),
 \qquad S_1\subsetneq\cdots\subsetneq S_k.
\]
It is incident at \(\mathbf s_0\) if
\(\Lambda_{S_i,\ell_i}(\mathbf s_0)=0\) for all \(i\).  It is
contributing on a parameter germ if its aggregate reduced coefficient
\(C_{\mathfrak F}\) in \eqref{eq:flag-normal-form} is not the zero germ there.
\end{definition}

\begin{proposition}[Nested sets of the maximal building set are chains]
\label{prop:maximal-incidence}
For the maximal building set of the Boolean support lattice, a nested
set is a strict support chain.  Consequently the canonical denominator
products in \eqref{eq:flag-normal-form} are indexed by strict decorated flags, and
no such product contains two incomparable supports.  A simultaneous
degeneration involving incomparable supports is represented, after
refinement, by one or more strict-chain terms through a larger support
containing their join \(S\cup T\); any mixed coefficient in the
original dependent carrier coordinates is the aggregate of those chain
terms.
\end{proposition}

\begin{proof}
The maximal building set is the full Boolean lattice with its minimal
element removed.  The nested-set definition says that the join of any
incomparable subfamily must not belong to the building set.  Here every
nonempty join belongs to it, so no incomparable pair is nested.  The
face-poset theorem for wonderful models identifies the resulting
boundary strata with these nested sets
\cite{FeichtnerKozlov2004,DeConciniProcesi1995}.

The last statement concerns the canonical maximal-model expansion, not
an arbitrary partial-fraction presentation in the original parameter
coordinates.  A coarser chart may display incomparable boundary pieces,
but after refinement no single canonical denominator contains both:
their common degeneration is distributed among strict chains through
the join divisor.  Whether a mixed coefficient in the original carrier
coordinates survives is decided only after these chain contributions
are aggregated, or equivalently by an iterated residue when the relevant
carrier forms are independent.
\end{proof}

For a flag choose \(i_h\in S_h\setminus S_{h-1}\), with
\(S_0=\varnothing\), and use the simple-coroot exponents
\(s_{\beta_{i_h}}\) as transverse coordinates.  Then
\[
 \frac{\partial\Lambda_{S_j,\ell_j}}
      {\partial s_{\beta_{i_h}}}
 =\begin{cases}0,&j<h,\\1,&j\ge h.
 \end{cases}
\]
The triangular matrix has determinant one.  Thus the canonical
multi-residue of a flag is independent of the order in which these
normal coordinates are extracted.

\subsection{Relative face operators and flag residues}

Let \(U\subsetneq V\) be consecutive supports in a flag and put
\(\delta=\ell_V-\ell_U\).  At that stage the weighted complementary
input has a finite homogeneous decomposition
\[
 P(Ru,n)=\sum_{j,\mu}R^j
 P^{\rm in}_{j,\mu}(u)P^{\rm out}_{j,\mu}(n).
\]
Write \(C(U,V)\) for the crossing forms of the relative layer.  Expanding
such a form as \((RA_\alpha+B_\alpha)^{-s_\alpha}\), a Taylor
multiindex \(\mathbf k\) contributes the radial degree
\(j-|\mathbf k|\).  The relative carrier identity gives the balance
condition
\[
 |\mathbf k|=j+\delta.
\]
This yields the finite operator
\begin{equation}\label{eq:relative-operator}
\begin{aligned}
 \mathfrak R_{U,V,\delta}^{\mathbf s}[P]
 ={}&\sum_{j,\mu}
 \sum_{\substack{\mathbf k\in\Nzero^{C(U,V)}\\
                  |\mathbf k|=j+\delta}}
 (-1)^{|\mathbf k|}
 \prod_{\alpha\in C(U,V)}
 \frac{(s_\alpha)_{k_\alpha}}{k_\alpha!}\\
 &\qquad\times
 \cP^{\rm red}_{U,V;j,\mu,\mathbf k}(\mathbf s)
 \cW_{U,V;j,\mu,\mathbf k}(\mathbf s).
\end{aligned}
\end{equation}
We call this finite map the relative face operator for the two
consecutive supports.  Terms with \(j+\delta<0\) are absent.  The
projective factor is reduced recursively before specialization, and the
complementary factor is the polynomially weighted KMT series inherited
by the next stage.

\begin{theorem}[Recursive flag-residue formula]
\label{thm:flag-residue}
For every incident strict decorated flag \(\mathfrak F\), its canonical
flag residue is the finite expression obtained by applying
\eqref{eq:relative-operator} successively, beginning with the residue
formula on a proper support \eqref{eq:residue-general}.  Equivalently it is the aggregate
coefficient \(C_{\mathfrak F}\) in the maximal-model local form
\eqref{eq:flag-normal-form}.  All projective-link poles are assigned to proper
flag extensions before the coefficient is evaluated.
\end{theorem}

\begin{proof}
Take the residue at the first carrier.  Every summand is a canonical
projective period multiplied by a weighted complementary series.
Restrict the next carrier form to the first hyperplane and repeat the
Mellin--Barnes extraction.  The radial balance is exactly
\(|\mathbf k|=j+\delta\), so the result is
\eqref{eq:relative-operator}.  Iteration terminates after the length of
the flag.  Recursive child subtraction makes the remaining link
coefficients holomorphic, and uniqueness of the local normal form gives
the aggregate coefficient \(C_{\mathfrak F}\).
\end{proof}

\subsection{Component gamma defects}

Suppose a relative layer separates into connected components \(\nu\).
Let \(d_\nu\) be the number of relative coordinates,
\(\sigma_\nu(\mathbf s)\) the sum of the relevant root exponents, and
\(\eta_\nu\) the net homogeneous insertion degree in that component.
The component-mass integration gives
\begin{equation}\label{eq:component-gamma}
 D(\mathbf s)=
 \frac{\prod_\nu\Gamma(A_\nu(\mathbf s))}
      {\Gamma(A(\mathbf s))},
 \qquad
 A_\nu=d_\nu+\eta_\nu-\sigma_\nu,
 \quad A=\sum_\nu A_\nu.
\end{equation}
We call this quotient the component gamma factor, and its vanishing
along a flag stratum a gamma defect.  For
consecutive decorated supports in a flag, radial balance gives
\begin{equation}\label{eq:flag-balance}
 A(\mathbf s)=
 -\bigl(\Lambda_{V,\ell_V}(\mathbf s)
        -\Lambda_{U,\ell_U}(\mathbf s)\bigr).
\end{equation}
Hence the total gamma argument vanishes on the flag stratum.  If all
numerator arguments avoid nonpositive integers, the reciprocal gamma
factor produces a simple zero.  If
\(A(t)=-M+ut+O(t^2)\) along a transverse curve and all numerator gamma
values are finite, then
\begin{equation}\label{eq:gamma-derivative}
 D'(0)=(-1)^M M!\,u\prod_\nu\Gamma(A_\nu(0)).
\end{equation}
If one numerator argument also satisfies
\(A_j(t)=-m+vt+O(t^2)\) with \(v\ne0\), then
\begin{equation}\label{eq:gamma-escape}
 \lim_{t\to0}D(t)=
 (-1)^{M+m}\frac{M!}{m!}\frac{u}{v}
 \prod_{\nu\ne j}\Gamma(A_\nu(0)).
\end{equation}
Equation \eqref{eq:gamma-escape} removes the gamma zero; it does not by
itself prove that the full reduced coefficient is nonzero, because
projective, complementary-zeta, or arithmetic-moment factors may still
vanish.

\subsection{Aggregate Laurent coefficients}

Let \(\gamma(t)=\mathbf s_0+t\mathbf v\) be an affine curve through a
parameter point, transverse to every incident carrier, and write
\[
 c_w=(\Lambda_w\circ\gamma)'(0)\ne0.
\]
For an incident flag \(\mathfrak F\), put
\(\delta_{\mathfrak F}=\operatorname{ord}_{t=0}
 C_{\mathfrak F}(\gamma(t))\), with value \(+\infty\) for the zero germ.

\begin{theorem}[Aggregate Laurent coefficients and pole order]
\label{thm:aggregate-order}
For every integer \(m\ge1\), define
\begin{equation}\label{eq:aggregate-Am}
 \mathcal A_m(\gamma)=
 \sum_{|\mathfrak F|\ge m}
 \frac{
  \left.\dfrac{d^{|\mathfrak F|-m}}{dt^{|\mathfrak F|-m}}
  C_{\mathfrak F}(\gamma(t))\right|_{t=0}}
 {(|\mathfrak F|-m)!\prod_{w\in\mathfrak F}c_w}.
\end{equation}
Then
\[
 [t^{-m}]Z_\Phi(\gamma(t))=\mathcal A_m(\gamma),
 \qquad
 \operatorname{pord}_{t=0}Z_\Phi(\gamma(t))
 =\max\{m:\mathcal A_m(\gamma)\ne0\}.
\]
Here \(\max\emptyset:=0\); when every \(\mathcal A_m(\gamma)\) with
\(m\ge1\) vanishes, the pullback extends holomorphically across \(t=0\)
and the pole order is zero.
In particular,
\begin{equation}\label{eq:defect-upper}
 \operatorname{pord}_{t=0}Z_\Phi(\gamma(t))
 \le
 \max_{\mathfrak F}
 (|\mathfrak F|-\delta_{\mathfrak F})_+.
\end{equation}
Equality in \eqref{eq:defect-upper} holds exactly when the aggregate
\(\mathcal A_m\) at the maximal candidate order is nonzero.
\end{theorem}

\begin{proof}
Pull back \eqref{eq:flag-normal-form}.  A flag of length \(k\) contributes
\(t^{-k}(\prod c_w)^{-1}C_{\mathfrak F}(\gamma(t))\).  Taylor expansion
of its holomorphic reduced coefficient gives
\eqref{eq:aggregate-Am}.  Summing all flags before taking the largest
nonzero coefficient proves the order statement and the upper bound.
\end{proof}

The maximum in \eqref{eq:defect-upper} is not generally an equality:
terms arising from different flags can cancel after restriction to the
same one-dimensional slice.  Thus the canonical flag normal form and an
arbitrary partial-fraction expression in dependent parameter
coordinates must not be conflated.

On a bounded parameter neighborhood only finitely many carrier levels
and flags occur.  After excluding factor poles, every
\(\mathcal A_m\) is holomorphic.  The order is therefore locally
constant away from the locally finite analytic union of the zero sets
of these aggregate functions.  This is an analytic jump-set statement,
not a claim of Zariski constructibility.

\section{Type \texorpdfstring{$A_2$}{A2}: residues and intersections}
\label{sec:A2}

Put
\begin{equation}\label{eq:Tornheim}
 T(a,b,c)=\sum_{m,n\ge1}m^{-a}n^{-b}(m+n)^{-c}.
\end{equation}
The exact singular locus is classical
\cite{MatsumotoTsumura2006,MNOST2008}.  The root variables are ordered
as
\[
 (a,b,c)\longleftrightarrow
 (\alpha_1,\alpha_2,\alpha_1+\alpha_2).
\]

\begin{proposition}[Complete \(A_2\) residue functions]\label{prop:A2}
For every \(\ell\ge0\),
\begin{align}
 \Res_{a+c=1-\ell}T(a,b,c)
 &=(-1)^\ell\frac{(c)_\ell}{\ell!}\,\zeta(b-\ell),
 \label{eq:A2-left}\\
 \Res_{b+c=1-\ell}T(a,b,c)
 &=(-1)^\ell\frac{(c)_\ell}{\ell!}\,\zeta(a-\ell).
 \label{eq:A2-right}
\end{align}
The full-support residue is
\begin{equation}\label{eq:A2-full}
 \Res_{a+b+c=2}T(a,b,c)=B(1-a,1-b).
\end{equation}
These are all generic polar divisors and they are generically simple.
\end{proposition}

\begin{proof}
For \(m\to\infty\),
\[
 (m+n)^{-c}
 =
 m^{-c}\sum_{\ell\ge0}
 (-1)^\ell\frac{(c)_\ell}{\ell!}\frac{n^\ell}{m^\ell}.
\]
The pole of \(\zeta(a+c+\ell)\) gives \eqref{eq:A2-left};
the other formula is symmetric.  Simultaneous scaling
\(m+n=r\) gives \eqref{eq:A2-full}.
\end{proof}

\begin{corollary}[Exact \(A_2\) zero loci]\label{cor:A2-zero}
On \(a+c=1-\ell\), away from poles of the displayed factors,
the set-theoretic zero locus of the residue is
\[
 \bigcup_{j=0}^{\ell-1}\{c=-j\}
 \ \cup
 \{\,\zeta(b-\ell)=0\,\}.
\]
For \(\ell=0\), the first union is empty.  The trivial-zero part is the
explicit family
\[
 b=\ell-2m,\qquad m=1,2,\ldots,
\]
while the remaining components are the factual affine hyperplanes
\(b=\ell+\rho\), where \(\rho\) ranges over nontrivial zeros of
\(\zeta\); no assertion about their distribution is made.
The right family has the symmetric description with \(a\) in place of
\(b\).
\end{corollary}

\begin{table}[ht]
\centering
\small
\begin{tabular}{@{}lll@{}}
\toprule
Published \(A_2\) family & carrier support & formula in this paper\\
\midrule
\(a+c=1-\ell\) & \(S=\{1\}\) & \eqref{eq:A2-left}\\
\(b+c=1-\ell\) & \(S=\{2\}\) & \eqref{eq:A2-right}\\
\(a+b+c=2\) & \(S=\{1,2\}\) & \eqref{eq:A2-full}\\
\bottomrule
\end{tabular}
\caption{Item-by-item \(A_2\) regression against the classical
Mordell--Tornheim theorem.}
\label{tab:A2-comparison}
\end{table}

\subsection{Depth two Euler--Zagier thinning}

The depth-two Euler--Zagier function is
\[
 \EZ,2(u,v)=T(u,0,v).
\]
On \(u+v=1-\ell\), formula \eqref{eq:A2-left} becomes
\[
 (-1)^\ell\frac{(v)_\ell}{\ell!}\zeta(-\ell).
\]
Positive even \(\ell\) disappear because \(\zeta(-2m)=0\).
For \(\ell=0\) and positive odd \(\ell\), the residue is generically
nonzero.  The other singleton family specializes to \(v=1-\ell\);
for every \(\ell\ge1\),
\[
 (1-\ell)_\ell=0.
\]
Together with the full line, this yields
\[
 v=1,\qquad
 u+v\in\{2,1,0,-2,-4,-6,\ldots\},
\]
the Akiyama--Egami--Tanigawa list \cite{AET2001}.

\subsection{A codimension-two calculation}

\begin{example}[The intersection \(H_{\{1\},0}\cap H_{\{1,2\},0}\)]
\label{ex:A2-intersection}
Set
\[
 x=a+c-1,\qquad y=a+b+c-2,
\]
and take a generic point of \(x=y=0\), away from the second singleton
family.  In the Mellin--Barnes representation, the moving pole
\(z=-x\) contributes
\[
 \frac{\Gamma(c-x)\Gamma(x)}{\Gamma(c)}\,\zeta(1+y)
 =
 \frac1{xy}+O(x^{-1})+O(y^{-1})+O(1).
\]
Independently, take the residue first on the full-support divisor
\(y=0\).  There
\[
 B(1-a,1-b)=B(c-x,x)=\frac1x+O(1),
\]
so division by the transverse form \(y\) gives the same double
coefficient.  Therefore the intersection has exact generic order two
and
\[
 T(a,b,c)
 =
 \frac1{xy}
 +O(x^{-1})+O(y^{-1})+O(1).
\]
This calculation is the basic consistency check for the
stratification in Theorem~\ref{thm:A2-strata}.
\end{example}

\begin{theorem}[Complete generic \(A_2\) codimension-two table]
\label{thm:A2-strata}
Put
\[
 A_\ell=a+c-1+\ell,\qquad
 B_m=b+c-1+m,\qquad
 F=a+b+c-2.
\]
At generic points of the indicated intersections,
\begin{align}
 [A_\ell^{-1}B_m^{-1}]T&=0,\label{eq:A2-AB}\\
 [A_\ell^{-1}F^{-1}]T
 &=(-1)^\ell\frac{(c)_\ell}{\ell!},\label{eq:A2-AF}\\
 [B_m^{-1}F^{-1}]T
 &=(-1)^m\frac{(c)_m}{m!}.\label{eq:A2-BF}
\end{align}
The three carrier families meet at
\[
 (a,b,c)=(1+m,1+\ell,-\ell-m).
\]
At that point the two nested mixed coefficients are
\[
 [A_\ell^{-1}F^{-1}]T
 =[B_m^{-1}F^{-1}]T
 =\binom{\ell+m}{\ell},
\]
whereas neither \(A_\ell^{-1}B_m^{-1}\) nor
\(A_\ell^{-1}B_m^{-1}F^{-1}\) occurs.  Simple-pole terms belonging to
the three individual divisors may of course also be present.
\end{theorem}

\begin{proof}
Take the residue on \(A_\ell=0\).  Its only possible pole along the
full-support pullback comes from
\(\zeta(b-\ell)\), and
\(F-A_\ell=b-1-\ell\).  This proves
\eqref{eq:A2-AF}; the right formula is symmetric.  Along
\(A_\ell=B_m=0\) neither singleton residue is generically singular,
so \eqref{eq:A2-AB} follows.  At the triple point,
\[
 (-1)^\ell\frac{(-\ell-m)_\ell}{\ell!}
 =\frac{(m+1)_\ell}{\ell!}
 =\binom{\ell+m}{\ell}.
\]
Finally the maximal support building set has no flag containing the two
incomparable singleton supports, and the support lattice has no strict
flag of length three.  Hence no triple denominator product occurs.
\end{proof}

\section{Type \texorpdfstring{$A_3$}{A3}: residues and flag coefficients}

Use the notation
\begin{equation}\label{eq:A3}
\begin{aligned}
 Z_{A_3}(a,b,c,d,e,f)
 =\sum_{m,n,p\geq1}&m^{-a}n^{-b}p^{-c}(m+n)^{-d}\\
 &\times(n+p)^{-e}(m+n+p)^{-f}.
\end{aligned}
\end{equation}
Matsumoto and Tsumura already proved meromorphic continuation and
determined this singular locus \cite{MatsumotoTsumura2006}.  The next
statement is an independent, face-symmetric reorganization of the locus
and its generic residues.

Let $T$ denote the meromorphic Mordell--Tornheim function
\eqref{eq:Tornheim}.

\begin{theorem}[The $A_3$ divisor and residue functions]\label{thm:A3}
The generic polar divisor of \eqref{eq:A3} consists of
\begin{align*}
 a+d+f&=1-\ell,&
 b+d+e+f&=1-\ell,&
 c+e+f&=1-\ell,\\
 a+b+d+e+f&=2-\ell,&
 b+c+d+e+f&=2-\ell,&
 a+c+d+e+f&=2-\ell,
\end{align*}
for $\ell\in\Nzero$, together with
\[
 a+b+c+d+e+f=3.
\]
The singleton residues are
\begin{align}
 \cR_{1,\ell}
 &=(-1)^\ell\!\sum_{u+v=\ell}
   \frac{\RF d u\RF f v}{u!v!}
   T(b-u,c,e-v),\label{eq:A3-R1}\\
 \cR_{3,\ell}
 &=(-1)^\ell\!\sum_{u+v=\ell}
   \frac{\RF e u\RF f v}{u!v!}
   T(a,b-u,d-v),\label{eq:A3-R3}\\
 \cR_{2,\ell}
 &=(-1)^\ell\!\sum_{u+v+w=\ell}
   \frac{\RF d u\RF e v\RF f w}{u!v!w!}
   T(a-u,c-v,-w).
 \label{eq:A3-R2}
\end{align}
The two-node residues are
\begin{align}
 \cR_{12,\ell}
 &=(-1)^\ell\zeta(c-\ell)
   \sum_{u+v=\ell}
   \frac{\RF e u\RF f v}{u!v!}
   B(1-a,1-b-e-u),\label{eq:A3-R12}\\
 \cR_{23,\ell}
 &=(-1)^\ell\zeta(a-\ell)
   \sum_{u+v=\ell}
   \frac{\RF d u\RF f v}{u!v!}
   B(1-b-d-u,1-c),\label{eq:A3-R23}\\
 \cR_{13,\ell}
 &=(-1)^\ell\zeta(b-\ell)
   \sum_{u+v+w=\ell}
   \frac{\RF d u\RF e v\RF f w}{u!v!w!}
   B(1-a-d-u,1-c-e-v).
 \label{eq:A3-R13}
\end{align}
Finally,
\begin{equation}\label{eq:A3-full}
 \Res_{a+b+c+d+e+f=3}Z_{A_3}
 =\FP\!\int_{\substack{x,y,z>0\\x+y+z=1}}
 x^{-a}y^{-b}z^{-c}(x+y)^{-d}(y+z)^{-e}\,\dd x\dd y.
\end{equation}
Every displayed residue is a nonzero meromorphic function on its
hyperplane.
\end{theorem}

\begin{proof}
Apply Theorem~\ref{thm:residue}.  For example, on $S=\{1\}$ the crossing
forms are $m+n$ and $m+n+p$; their Taylor indices $u,v$ leave the
weighted complement function $T(b-u,c,e-v)$, proving
\eqref{eq:A3-R1}.  On $S=\{1,2\}$ the crossing forms are $n+p$ and
$m+n+p$.  Their total external power is $p^\ell$, while the projective
integral is $B(1-a,1-b-e-u)$, proving \eqref{eq:A3-R12}.  The other
formulas follow by the same calculation.  Generic nonvanishing is
Theorem~\ref{thm:true-locus}.
\end{proof}

\subsection{Unshifted flags and the first order-three coefficient}

At level zero, the seven nonempty supports of \(A_3\) have twelve
comparable pairs, nine incomparable pairs, and six complete support
flags.  Taking residues of
\eqref{eq:A3-R1}--\eqref{eq:A3-full} gives the following complete
unshifted two-flag table.

\begin{table}[ht]
\centering
\small
\begin{tabular}{@{}ll@{\qquad}ll@{}}
\toprule
Flag & coefficient & Flag & coefficient\\
\midrule
\(\{1\}\subset\{1,2\}\)&\(\zeta(c)\)&
\(\{1\}\subset\{1,3\}\)&\(\zeta(b)\)\\
\(\{1\}\subset I\)&\(B(1-b,1-c)\)&
\(\{2\}\subset\{1,2\}\)&\(\zeta(c)\)\\
\(\{2\}\subset\{2,3\}\)&\(\zeta(a)\)&
\(\{2\}\subset I\)&\(0\)\\
\(\{3\}\subset\{1,3\}\)&\(\zeta(b)\)&
\(\{3\}\subset\{2,3\}\)&\(\zeta(a)\)\\
\(\{3\}\subset I\)&\(B(1-a,1-b)\)&
\(\{1,2\}\subset I\)&\(B(1-a,1-b-e)\)\\
\(\{2,3\}\subset I\)&\(B(1-b-d,1-c)\)&
\(\{1,3\}\subset I\)&\(B(1-a-d,1-c-e)\)\\
\bottomrule
\end{tabular}
\caption{All twelve unshifted two-flags in type \(A_3\).}
\label{tab:A3-flags}
\end{table}

The zero in the middle row is structural.  On \(H_{\{2\},0}\), the
residue is
\[
 T(a,c,0)=\zeta(a)\zeta(c),
\]
while the full-support pullback is \(a+c=2\), which is not a polar
hyperplane of this product at a generic point.

\begin{proposition}[Complete unshifted \(A_3\) flags]
\label{prop:A3-complete-flags}
Every complete support flag
\[
 \{i\}\subset J\subset I,
 \qquad |J|=2,
\]
has canonical order-three coefficient \(1\).  In particular, for
\(\{1\}\subset\{1,2\}\subset I\), the successive residues are
\[
 T(b,c,e)\longmapsto\zeta(c)\longmapsto1.
\]
The reflected and middle flags are analogous.
\end{proposition}

For later comparison, the shifted triple flag
\[
 H_{\{1\},\ell}\subset H_{\{1,2\},m}\subset H_{I,0}
\]
has coefficient
\begin{equation}\label{eq:A3-shifted-triple}
 \frac{(-1)^{\ell+m}}{m!}
 \sum_{u+v=\ell}
 \frac{(d)_u(f)_v(e-v)_m}{u!v!}.
\end{equation}
At \(\ell=m=0\) this is \(1\).  For a summand with \(u+v=\ell\), the ambient relation
\(b+e=1-(m-\ell)\) becomes
\((b-u)+(e-v)=1-m\) in its Mordell--Tornheim factor.  Taking that
level-\(m\) residue and then the \(c=1+m\) residue of the remaining zeta
factor gives \eqref{eq:A3-shifted-triple}.

\begin{proposition}[Disjoint and disconnected probes]
\label{prop:A3-probes}
At a generic point of
\(H_{\{1\},0}\cap H_{\{3\},0}\), the function \(Z_{A_3}\) has no
\(\Lambda_{\{1\},0}^{-1}\Lambda_{\{3\},0}^{-1}\) term.  The same is
true for the adjacent incomparable singleton pairs.

The flag \(H_{\{2\},0}\subset H_{I,0}\) is the smallest
multivariable example of a disconnected relative gamma defect.  Its
component-mass quotient is
\[
 B(1-a,1-c)=
 \frac{\Gamma(1-a)\Gamma(1-c)}{\Gamma(2-a-c)},
\]
which vanishes generically on \(a+c=2\).  The loci
\(1-a\in\mathbb Z_{\le0}\) or \(1-c\in\mathbb Z_{\le0}\) remove the
gamma zero, but do not alone guarantee a nonzero full flag coefficient.
\end{proposition}

\begin{proof}
On \(H_{\{1\},0}\), the residue is \(T(b,c,e)\).  The pullback of
\(H_{\{1,2\},0}\) is \(b+e=1\), a genuine Mordell--Tornheim pole,
whereas the pullbacks of the two incomparable singleton carriers are
not generic polar equations of \(T(b,c,e)\).  This proves the first
claim by iterated residues.  The disconnected statement follows from
\eqref{eq:component-gamma}; the denominator argument is zero on the
flag, while the numerator arguments are generically nonintegral.
\end{proof}

\subsection{The \texorpdfstring{\(A_3\)}{A3} singular hyperplanes}

Matsumoto and Tsumura use
\[
 (s_1,s_2,s_3,s_4,s_5,s_6)=(a,b,c,d,e,f).
\]
Their complete \(\mathfrak{sl}(4)\) list matches the carrier theorem
term by term:
\begin{table}[ht]
\centering
\small
\begin{tabular}{@{}lll@{}}
\toprule
Published hyperplane & support & equation here\\
\midrule
\(s_1+s_4+s_6=1-\ell\) & \(\{1\}\) & \(a+d+f=1-\ell\)\\
\(s_3+s_5+s_6=1-\ell\) & \(\{3\}\) & \(c+e+f=1-\ell\)\\
\(s_2+s_4+s_5+s_6=1-\ell\) & \(\{2\}\) & \(b+d+e+f=1-\ell\)\\
\(s_1+s_2+s_4+s_5+s_6=2-\ell\) & \(\{1,2\}\) &
 \(a+b+d+e+f=2-\ell\)\\
\(s_1+s_3+s_4+s_5+s_6=2-\ell\) & \(\{1,3\}\) &
 \(a+c+d+e+f=2-\ell\)\\
\(s_2+s_3+s_4+s_5+s_6=2-\ell\) & \(\{2,3\}\) &
 \(b+c+d+e+f=2-\ell\)\\
\(\sum_{i=1}^6s_i=3\) & full & \(a+b+c+d+e+f=3\)\\
\bottomrule
\end{tabular}
\caption{Theorem~\ref{thm:A3} against the 2006 \(A_3\) theorem.}
\label{tab:A3-comparison}
\end{table}

\subsection{Depth-three Euler--Zagier specialization}

We use Zhao's increasing-index convention
\[
 \zeta_{\mathrm{EZ},d}(s_1,\ldots,s_d)
 =\sum_{0<n_1<\cdots<n_d}
   n_1^{-s_1}\cdots n_d^{-s_d},
 \qquad \zeta_{\mathrm{EZ},0}=1.
\]
With the gap variables
\(m=n_1\), \(n=n_2-n_1\), and \(p=n_3-n_2\), the standard embedding is
\begin{equation}\label{eq:EZ3map}
 \iota(s_1,s_2,s_3)
 =(a,b,c,d,e,f)=(s_1,0,0,s_2,0,s_3).
\end{equation}
Put
\[
 t=s_2+s_3,\qquad w=s_1+s_2+s_3.
\]
For an integer \(M\ge2\), the Euler--Maclaurin recurrence of
Akiyama--Egami--Tanigawa may be written
\begin{equation}\label{eq:AET-rec}
\begin{aligned}
 \EZ,3(s_1,s_2,s_3)
={}&\frac{\EZ,2(s_1,t-1)}{s_3-1}
 -\frac12\EZ,2(s_1,t)\\
 &+\sum_{q=1}^{M-1}
 \frac{(s_3)_qB_{q+1}}{(q+1)!}\,
 \EZ,2(s_1,t+q)
 +H_M(\mathbf s),
\end{aligned}
\end{equation}
where \(H_M\) is holomorphic in the enlarged region obtained after
\(M\) shifts.

\begin{theorem}[Depth-three specialization]\label{thm:EZ3}
The singular hyperplanes of \(\EZ,3\) are exactly
\[
 s_3=1,\qquad
 t\in\{2,1,0,-2,-4,-6,\ldots\},\qquad
 w\in\mathbb Z_{\le3}.
\]
All are generically simple.  This is exactly the depth-three instance
of the Akiyama--Egami--Tanigawa theorem.
\end{theorem}

\begin{proof}
The denominator in the first term of \eqref{eq:AET-rec} gives
\(s_3=1\), with generic residue \(\EZ,2(s_1,s_2)\).
Equivalently, the \(A_3\) \(S=\{3\}\) residue has the factor
\((s_3)_\ell\) on \(s_3=1-\ell\), so every \(\ell\ge1\) line vanishes.

For the last two variables, \(t=2\) comes from the vertical pole of
\(\EZ,2(s_1,t-1)\).  The term \(q=0\) gives \(t=1\), and the term
indexed by \(q\ge1\) gives
\[
 t=1-q,\qquad
 \Res=\frac{(s_3)_qB_{q+1}}{(q+1)!}
\]
up to the nonzero unit residue of the depth-two vertical line.
For even \(q\ge2\), \(B_{q+1}=0\); for \(q=0\) and odd \(q\), the
coefficient is generically nonzero.  Hence the last-two list is
\[
 2,1,0,-2,-4,\ldots.
\]
In the \(A_3\) carrier language, the families \(S=\{2\}\) and
\(S=\{2,3\}\) coalesce under \(\iota\); equation
\eqref{eq:AET-rec} is the incidence-complete sum and exhibits its
Bernoulli cancellation.

It remains to explain why the deepest family has no gaps.  On
$w=3$, the first term in \eqref{eq:AET-rec} contributes the nonzero
generic residue
\[
 \frac{1}{(s_3-1)(1-s_1)}.
\]
On $w=2$, the first two terms contribute
\[
 -\frac{1}{2(s_3-1)}-\frac{1}{2(1-s_1)},
\]
which is not the zero germ; no $q\ge1$ term is singular there.  For
$w\le1$, choose
\[
 q_*(w)=
 \begin{cases}
 2-w,&w\ \text{odd},\\
 1-w,&w\ \text{even}.
 \end{cases}
\]
Then $q_*(w)$ is a nonnegative odd integer, so
$B_{q_*(w)+1}\ne0$.  If $w$ is odd, the selected depth-two factor lies
on its full-support line $s_1+(t+q_*)=2$, whose generic residue is
$B(1-s_1,1)=1/(1-s_1)$.  If $w$ is even, it lies on the shifted
singleton line $s_1+(t+q_*)=1$, whose generic residue is
$\zeta(0)=-1/2$.  Both are nonzero germs.  Moreover
$(s_3)_{q_*(w)}$ has strictly larger degree in $s_3$ than every other
term contributing to the same total-sum line.  Hence its leading
coefficient cannot be cancelled, and the aggregate residue is not the
zero germ.  This proves every integer line $w\le3$ and explains the
asymmetry: the last-two family is controlled by one Bernoulli
coefficient, whereas the deepest family is an incidence sum with a
unique highest-degree nonzero term.
\end{proof}

\begin{table}[ht]
\centering
\small
\begin{tabular}{@{}lll@{}}
\toprule
Depth-three family & source in \eqref{eq:AET-rec} & thinning mechanism\\
\midrule
\(s_3=1\) & \(1/(s_3-1)\) & \((1-\ell)_\ell=0\) for shifts\\
\(t=2\) & shifted argument \(t-1=1\) & survives\\
\(t=1-q\) & \(B_{q+1}(s_3)_q\) & \(B_{2m+1}=0\), \(m\ge1\)\\
\(w\le3\) & several incident depth-two lines & highest-degree term survives\\
\bottomrule
\end{tabular}
\caption{Depth-three specialization of the carrier families.}
\label{tab:EZ3}
\end{table}

\subsection{Zhao's residue formulas}
\label{subsec:Zhao-comparison}

We compare the specialization calculus with Zhao's pole and residue
theorems \cite[Theorems~5 and~7]{Zhao2000}.  For depth \(d\), write
\[
 s_d(j)=s_j+\cdots+s_d,\qquad
 u_j=s_d(j)-d+j-1,
\]
and, for nonnegative integers \(a_{i+1},\ldots,a_d\), put
\[
 a_d(j)=a_j+\cdots+a_d,\qquad a_d(d+1)=0.
\]
With Bernoulli convention \(B_1=-1/2\), Zhao's nonvertical residue on
\[
 s_d(i)=d-i+2-l\qquad(1\le i<d,\ l\ge1)
\]
is
\begin{equation}\label{eq:Zhao-residue}
 \zeta(s_1,\ldots,s_{i-1})
 \sum_{\substack{a_{i+1},\ldots,a_d\ge0\\
                  a_{i+1}+\cdots+a_d=l-1}}
 \prod_{j=i+1}^{d}
 \frac{B_{a_j}\,\Gamma(a_d(j)+u_j)}
      {a_j!\,\Gamma(a_d(j+1)+u_j+1)}.
\end{equation}
The vertical residue at \(s_d=1\) is
\(\zeta(s_1,\ldots,s_{d-1})\).  In depth two, Zhao's parameter
\(l=1\) gives the full line \(u+v=2\), whereas the line
\(u+v=1-\ell\) corresponds to \(l=\ell+2\).  This explicit index shift
is used in the comparison below.

For depth two, \eqref{eq:Zhao-residue} gives the following exact
comparison.
\begin{table}[ht]
\centering
\small
\begin{tabularx}{\textwidth}{@{}l X X@{}}
\toprule
hyperplane & Zhao's coefficient & specialization coefficient\\
\midrule
\(v=1\) & \(\zeta(u)\) & \(\zeta(u)\)\\
\(u+v=2\) & \(B_0\Gamma(v-1)/\Gamma(v)=1/(v-1)\)
 & \(B(1-u,1)=1/(v-1)\)\\
\(u+v=1-\ell\) &
 \(B_{\ell+1}(v)_\ell/(\ell+1)!\) &
 \((-1)^\ell(v)_\ell\zeta(-\ell)/\ell!\)\\
\bottomrule
\end{tabularx}
\caption{Depth-two residue comparison.  The last two expressions agree by
\(\zeta(-\ell)=(-1)^\ell B_{\ell+1}/(\ell+1)\).}
\label{tab:Zhao2}
\end{table}

For depth three, Zhao's formula gives
\[
 \Res_{s_3=1}\EZ,3=\EZ,2(s_1,s_2),
\]
\[
 \Res_{t=2}\EZ,3=\frac{\zeta(s_1)}{s_3-1},
 \qquad
 \Res_{t=1-q}\EZ,3
 =\zeta(s_1)\frac{B_{q+1}(s_3)_q}{(q+1)!}.
\]
For the deepest family \(w=4-l\), \eqref{eq:Zhao-residue} becomes the
finite convolution
\begin{equation}\label{eq:Zhao3-deep}
 \sum_{a_2+a_3=l-1}
 \frac{B_{a_2}\Gamma(a_2+a_3+u_2)}
      {a_2!\Gamma(a_3+u_2+1)}
 \frac{B_{a_3}\Gamma(a_3+u_3)}
      {a_3!\Gamma(u_3+1)}.
\end{equation}
Expanding the two depth-two factors in \eqref{eq:AET-rec}, grouping
terms by \(a_2+a_3\), and using
\(\Gamma(z+n)/\Gamma(z)=(z)_n\) gives
\eqref{eq:Zhao3-deep} term by term.  Thus the specialization calculus
recovers both Zhao's residue values and the Akiyama--Egami--Tanigawa
survival pattern.

\section{Rank-two examples}

Besides the trivial \(A_1\) function \(\zeta(s)\), the irreducible
rank-two types are \(A_2\), \(B_2=C_2\), and \(G_2\).  Type \(A_2\)
was treated in Section~\ref{sec:A2}.  We now fix the conventions for the remaining
types and display all generic residue functions.

\subsection{The \texorpdfstring{\(C_2\)}{C2} convention and its dictionary}

Set
\begin{equation}\label{eq:C2}
 Z_{C_2}(a,b,c,d)=
 \sum_{m,n\ge1}m^{-a}n^{-b}(m+n)^{-c}(m+2n)^{-d}.
\end{equation}
In the sequential four-variable notation used in the later rank-two
papers, \eqref{eq:C2} is simply ordered as $(a,b,c,d)$.  In the
root-indexed notation of Komori--Matsumoto--Tsumura
\cite{KMT2010II}, the ordered factors are
\[
 (m_1+m_2)^{-s_1}m_2^{-s_2}m_1^{-s_{12-}}
 (m_1+2m_2)^{-s_{12+}},
\]
so our dictionary is
\[
 (a,b,c,d)=(s_{12-},s_2,s_1,s_{12+}).
\]
The \(B_2\) convention is obtained by reversing the Dynkin arrow and
permuting the two simple coordinates.

\begin{proposition}[Complete \(C_2\) residue functions]\label{prop:C2}
The generic polar divisor is
\[
 a+c+d=1-\ell,\qquad
 b+c+d=1-\ell\quad(\ell\ge0),\qquad
 a+b+c+d=2.
\]
The singleton residues are
\begin{align}
 \Res_{a+c+d=1-\ell}Z_{C_2}
 &=(-1)^\ell\zeta(b-\ell)
 \sum_{u+v=\ell}
 \frac{(c)_u(d)_v}{u!v!}2^v,
 \label{eq:C2m}\\
 \Res_{b+c+d=1-\ell}Z_{C_2}
 &=(-1)^\ell2^{-d}\zeta(a-\ell)
 \sum_{u+v=\ell}
 \frac{(c)_u(d)_v}{u!v!}2^{-v}.
 \label{eq:C2n}
\end{align}
The full-support residue is
\begin{equation}\label{eq:C2full}
 2^{-d}B(1-a,1-b)
 {}_2F_1\!\left(d,1-a;2-a-b;\frac12\right).
\end{equation}
\end{proposition}

\begin{proof}
For \(m\to\infty\), expand
\[
 (m+n)^{-c}(m+2n)^{-d}
 =
 m^{-c-d}
 \sum_{\ell\ge0}(-1)^\ell
 \sum_{u+v=\ell}
 \frac{(c)_u(d)_v}{u!v!}2^v
 \frac{n^\ell}{m^\ell}.
\]
This gives \eqref{eq:C2m}; the \(n\)-face gives
\eqref{eq:C2n}.  On the full radial face, write \(m=r x\),
\(n=r(1-x)\).  Since \(m+2n=2r(1-x/2)\), Euler's integral for
\({}_2F_1\) gives \eqref{eq:C2full}.
\end{proof}

\begin{table}[ht]
\centering
\small
\begin{tabular}{@{}lll@{}}
\toprule
KMT \(C_2\) family & our equation & residue\\
\midrule
first singleton & \(a+c+d=1-\ell\) & \eqref{eq:C2m}\\
second singleton & \(b+c+d=1-\ell\) & \eqref{eq:C2n}\\
full support & \(a+b+c+d=2\) & \eqref{eq:C2full}\\
\bottomrule
\end{tabular}
\caption{The fixed-variant dictionary and the three published \(C_2\)
singularity families.}
\label{tab:C2-comparison}
\end{table}

\subsection{The \texorpdfstring{\(G_2\)}{G2} function}

Use the standard ordered form
\begin{equation}\label{eq:G2}
\begin{aligned}
 Z_{G_2}(\mathbf s)
 =\sum_{m,n\ge1}
 &m^{-s_1}n^{-s_2}(m+n)^{-s_3}(m+2n)^{-s_4}\\
 &\times(m+3n)^{-s_5}(2m+3n)^{-s_6}.
\end{aligned}
\end{equation}
For \(\ell\ge0\), define
\begin{align}
 A^{(1)}_\ell(\mathbf s)
 &=
 \sum_{k_3+k_4+k_5+k_6=\ell}
 \frac{(s_3)_{k_3}(s_4)_{k_4}(s_5)_{k_5}(s_6)_{k_6}}
      {k_3!k_4!k_5!k_6!}\,
 2^{k_4}3^{k_5}\left(\frac32\right)^{k_6},
 \label{eq:G2A1}\\
 A^{(2)}_\ell(\mathbf s)
 &=
 \sum_{k_3+k_4+k_5+k_6=\ell}
 \frac{(s_3)_{k_3}(s_4)_{k_4}(s_5)_{k_5}(s_6)_{k_6}}
      {k_3!k_4!k_5!k_6!}\,
 2^{-k_4}3^{-k_5}\left(\frac23\right)^{k_6}.
 \label{eq:G2A2}
\end{align}

\begin{proposition}[Complete generic \(G_2\) residues]\label{prop:G2}
The generic polar divisor is
\begin{align*}
 s_1+s_3+s_4+s_5+s_6&=1-\ell,\\
 s_2+s_3+s_4+s_5+s_6&=1-\ell
 \qquad(\ell\ge0),\\
 s_1+s_2+s_3+s_4+s_5+s_6&=2.
\end{align*}
The singleton residues are
\begin{align}
 \Res_{H_{1,\ell}}Z_{G_2}
 &=(-1)^\ell2^{-s_6}
 A^{(1)}_\ell(\mathbf s)\,\zeta(s_2-\ell),
 \label{eq:G2R1}\\
 \Res_{H_{2,\ell}}Z_{G_2}
 &=(-1)^\ell2^{-s_4}3^{-s_5-s_6}
 A^{(2)}_\ell(\mathbf s)\,\zeta(s_1-\ell).
 \label{eq:G2R2}
\end{align}
The full-support residue is
\begin{equation}\label{eq:G2full}
\begin{aligned}
 &2^{-s_4}3^{-s_5-s_6}B(1-s_1,1-s_2)\\
 &\quad\times
 \FD^{(3)}\!\left(
 1-s_1;\ s_4,s_5,s_6;\ 2-s_1-s_2;\,
 \frac12,\frac23,\frac13\right).
\end{aligned}
\end{equation}
\end{proposition}

\begin{proof}
The \(m\)-face factors
\(2m+3n=2m(1+3n/(2m))\), giving
\eqref{eq:G2A1} and \eqref{eq:G2R1}.  The \(n\)-face factors
\[
 m+2n=2n(1+m/(2n)),\quad
 m+3n=3n(1+m/(3n)),\quad
 2m+3n=3n(1+2m/(3n)),
\]
giving \eqref{eq:G2A2} and \eqref{eq:G2R2}.
For the full face put \(m=r x\), \(n=r(1-x)\); the three nontrivial
factors become \(1-x/2\), \(1-2x/3\), and \(1-x/3\).
The Euler integral for the Lauricella function gives
\eqref{eq:G2full}.
\end{proof}

Komori, Matsumoto and Tsumura proved meromorphic continuation and
listed exactly the three \(G_2\) families above as possible
singularities \cite{KMT2011G2}.  The generic-divisor theorem upgrades
that candidate list to genuine generic poles and
Proposition~\ref{prop:G2} supplies their residue functions.

\begin{table}[ht]
\centering
\small
\begin{tabular}{@{}llll@{}}
\toprule
Type & proper shifted families & full line & status\\
\midrule
\(A_1\) & none & \(s=1\) & classical\\
\(A_2\) & two & sum \(=2\) & exact identity\\
\(B_2=C_2\) & two & sum \(=2\) & exact identity\\
\(G_2\) & two & sum \(=2\) & candidate list upgraded generically\\
\bottomrule
\end{tabular}
\caption{Complete irreducible rank-two generic coverage in the fixed
variants.}
\label{tab:rank2}
\end{table}

\section{Types \texorpdfstring{\(B_3\)}{B3} and \texorpdfstring{\(C_3\)}{C3}}
\label{sec:B3C3}

\subsection{The nine-form polynomials and carrier data}

The \(B/C\) arrow convention is a recurrent source of errors.  We
therefore define the two objects by their nine linear forms, in the
normalization used by Au \cite[Section~9]{Au2025}.  Put
\begin{align}
 P_{B_3}(m_1,m_2,m_3)
={}&m_1m_2m_3(m_1+m_2)(m_2+m_3)(m_1+m_2+m_3)\notag\\
 &\times(2m_2+m_3)(2m_1+2m_2+m_3)(m_1+2m_2+m_3),
 \label{eq:B3poly}\\
 P_{C_3}(m_1,m_2,m_3)
={}&m_1m_2m_3(m_1+m_2)(m_2+m_3)(m_1+m_2+m_3)\notag\\
 &\times(m_2+2m_3)(m_1+m_2+2m_3)(m_1+2m_2+2m_3).
 \label{eq:C3poly}
\end{align}
The multivariable functions \(Z_{B_3}(\mathbf s)\) and
\(Z_{C_3}(\mathbf s)\) attach exponents \(s_1,\ldots,s_9\) to the
factors in the displayed order.  Their diagonal restrictions are
\(\xi_{B_3}(s)\) and \(\xi_{C_3}(s)\).  This displayed dictionary,
rather than an implicit Dynkin-arrow convention, fixes every formula
below.  The general \(B_r,C_r,D_r\) analytic framework is due to
Komori--Matsumoto--Tsumura \cite{KMT2010II}; their Part~III develops the
Weyl-functional-relation side \cite{KMT2012III}.

Both types have the same support-count ledger:
\begin{table}[ht]
\centering
\small
\begin{tabular}{@{}ccl@{}}
\toprule
support \(S\) & \(N(S)\) & carrier form \(\sum_{j\in R(S)}s_j\)\\
\midrule
\(\{1\}\)&5&\(s_1+s_4+s_6+s_8+s_9\)\\
\(\{2\}\)&7&\(s_2+s_4+s_5+s_6+s_7+s_8+s_9\)\\
\(\{3\}\)&6&\(s_3+s_5+s_6+s_7+s_8+s_9\)\\
\(\{1,2\}\)&8&\(\sum_{j=1}^9s_j-s_3\)\\
\(\{1,3\}\)&8&\(\sum_{j=1}^9s_j-s_2\)\\
\(\{2,3\}\)&8&\(\sum_{j=1}^9s_j-s_1\)\\
\(\{1,2,3\}\)&9&\(\sum_{j=1}^9s_j\)\\
\bottomrule
\end{tabular}
\caption{Exact \(B_3/C_3\) support-count ledger.}
\label{tab:B3C3-carrier}
\end{table}
Consequently every proper-support line in the table, shifted by
\(|S|-\ell\), is a genuine generic divisor; the only full-support
divisor is \(\sum_js_j=3\).  On the diagonal, the unshifted supports
produce the five positive values
\(1/3,1/4,1/5,1/6,1/7\).  The first shift of each two-node support also
produces the additional candidate \(1/8\); its aggregate cancellation is treated
separately below.

For completeness, the singleton Taylor systems are entirely explicit.
Write the \(j\)-th form as
\(L_j=a_{ji}m_i+B_{ji}(\widehat{\mathbf m_i})\), and put
\(J_i=\{j:a_{ji}>0\}\).  Then
\begin{equation}\label{eq:B3C3-singleton}
\begin{aligned}
 \Res_{H_{\{i\},\ell}}Z_\Phi(\mathbf s)
 ={}&(-1)^\ell
 \sum_{\substack{\mathbf k\in\Nzero^{J_i}\\|\mathbf k|=\ell}}
 \prod_{j\in J_i}
 \frac{(s_j)_{k_j}}{k_j!}\,a_{ji}^{-s_j-k_j}\\
 &\times
 \sum_{\widehat{\mathbf m_i}\in\N^2}
 \prod_{j\notin J_i}L_j(\widehat{\mathbf m_i})^{-s_j}
 \prod_{j\in J_i}B_{ji}(\widehat{\mathbf m_i})^{k_j}.
\end{aligned}
\end{equation}
Terms with \(B_{ji}=0\) and \(k_j>0\) vanish.  The coefficient lists
\((a_{ji})_{j\in J_i}\) are
\begin{equation}\label{eq:B3C3-singleton-coeffs}
\begin{array}{c|ccc}
& i=1&i=2&i=3\\ \hline
B_3&(1,1,1,2,1)&(1,1,1,1,2,2,2)&(1,1,1,1,1,1)\\
C_3&(1,1,1,1,1)&(1,1,1,1,1,1,2)&(1,1,1,2,2,2).
\end{array}
\end{equation}
The pair and full-support residue functions are the corresponding
one-dimensional and two-dimensional projective-period instances of
Theorem~\ref{thm:residue}.  Thus this section gives a complete
carrier-level specialization of the generic theory; where no closed
gamma reduction is known here, the residue remains an explicit
projective integral rather than an unevaluated formal symbol.

\begin{proposition}[Exact pair-face periods at \(s=1/4\)]
\label{prop:B3C3-pair-periods}
Let \(u\in(0,1)\) be the first coordinate on the two-node simplex.
The six pair-face polynomials obtained directly from
\eqref{eq:B3poly}--\eqref{eq:C3poly} are
\[
\begin{array}{c|ccc}
 &12&13&23\\
\hline
B_3&
4u(1-u)^3(2-u)&
u^2(1-u)^3(1+u)&
u^2(1-u)(1+u)^3\\[1mm]
C_3&
u(1-u)^3(2-u)&
2u^2(1-u)^3(2-u)^2&
2u^2(1-u)(2-u)^2.
\end{array}
\]
Writing
\[
 I^\Phi_{ij}=\int_0^1Q^\Phi_{ij}(u)^{-1/4}\,\dd u,
\]
Euler's integral for the Gauss function gives
\begin{align}
 I^{B_3}_{12}
 &=2^{-3/4}B\!\left(\frac34,\frac14\right)
 {}_2F_1\!\left(\frac34,\frac14;1;\frac12\right),\label{eq:B3I12}\\
 I^{B_3}_{13}
 &=B\!\left(\frac12,\frac14\right)
 {}_2F_1\!\left(\frac12,\frac14;\frac34;-1\right),\label{eq:B3I13}\\
 I^{B_3}_{23}
 &=B\!\left(\frac12,\frac34\right)
 {}_2F_1\!\left(\frac12,\frac34;\frac54;-1\right),\label{eq:B3I23}\\
 I^{C_3}_{12}
 &=2^{-1/4}B\!\left(\frac34,\frac14\right)
 {}_2F_1\!\left(\frac34,\frac14;1;\frac12\right),\label{eq:C3I12}\\
 I^{C_3}_{13}
 &=2^{-3/4}B\!\left(\frac12,\frac14\right)
 {}_2F_1\!\left(\frac12,\frac12;\frac34;\frac12\right),\label{eq:C3I13}\\
 I^{C_3}_{23}
 &=2^{-3/4}B\!\left(\frac12,\frac34\right)
 {}_2F_1\!\left(\frac12,\frac12;\frac54;\frac12\right).\label{eq:C3I23}
\end{align}
Hence the internally derived pair contribution is
\[
 \Res_{s=1/4}\xi_\Phi(s)
 =\frac{\zeta(1/4)}8
 \left(I^\Phi_{12}+I^\Phi_{13}+I^\Phi_{23}\right).
\]
Au's closed evaluations are equivalently the two exact sum identities
\begin{align}
 \sum_{ij\in\{12,13,23\}}I^{B_3}_{ij}
 &=
 \frac{2^{1/4}+\csc(\pi/8)}{4\sqrt\pi}\,
 \Gamma\!\left(\frac18\right)\Gamma\!\left(\frac38\right),
 \label{eq:B3-period-sum}\\
 \sum_{ij\in\{12,13,23\}}I^{C_3}_{ij}
 &=
 \frac{2+\csc(\pi/8)}{4\,2^{1/4}\sqrt\pi}\,
 \Gamma\!\left(\frac18\right)\Gamma\!\left(\frac38\right).
 \label{eq:C3-period-sum}
\end{align}
\end{proposition}

\begin{proof}
Restrict each product in \eqref{eq:B3poly}--\eqref{eq:C3poly} to the
relevant pair of variables and then impose \(u+(1-u)=1\).  This gives
the six displayed polynomials, including every coefficient \(2\).
Each integral is then an instance of
\[
 \int_0^1u^{A-1}(1-u)^{C-A-1}(1-zu)^{-B}\,\dd u
 =B(A,C-A)\,{}_2F_1(A,B;C;z).
\]
The residue factor \(1/8\) is the transverse derivative of the pair
carrier.  The final two sum identities are Au's independent
evaluations of the same residue, rewritten after multiplication by
\(8/\zeta(1/4)\) \cite[Theorems~9.4 and~9.6]{Au2025}.
\end{proof}

\subsection{Positive diagonal residues}

For a two-node support \(S=\{i,j\}\), define the diagonal projective
period
\begin{equation}\label{eq:B3C3-pairperiod}
 \cP_{\Phi,S}(q)
 =\int_0^1
 \prod_{k:\,\supp L_k\cap S\ne\varnothing}
 \bigl(b_{ki}u+b_{kj}(1-u)\bigr)^{-q}\,\dd u.
\end{equation}
The all-type face formula gives
\begin{equation}\label{eq:B3C3-quarter-period}
 \Res_{s=1/4}\xi_\Phi(s)
 =\frac{\zeta(1/4)}8
 \sum_{|S|=2}\cP_{\Phi,S}(1/4).
\end{equation}
Proposition~\ref{prop:B3C3-pair-periods} derives the six pair periods
exactly as Euler hypergeometric integrals, while Au independently
evaluates their sums in closed gamma form.  The resulting positive
residues are listed in Table~\ref{tab:B3C3-positive}.
\begin{table}[ht]
\centering
\small
\begin{tabularx}{\textwidth}{@{}c X X@{}}
\toprule
pole & \(B_3\) residue & \(C_3\) residue\\
\midrule
\(1/3\)&\multicolumn{2}{c}{\(\Gamma(1/3)^6/(96\pi^2)\)}\\
\(1/4\)&
\(\dfrac{2^{1/4}+\csc(\pi/8)}{32\sqrt\pi}
 \Gamma(1/8)\Gamma(3/8)\zeta(1/4)\)&
\(\dfrac{2+\csc(\pi/8)}{32\,2^{1/4}\sqrt\pi}
 \Gamma(1/8)\Gamma(3/8)\zeta(1/4)\)\\
\(1/5\)&\(\dfrac{2^{-1/5}}5\xi_{B_2}(1/5)\)&
\(\dfrac15\xi_{B_2}(1/5)\)\\
\(1/6\)&\(\dfrac16\xi_{A_2}(1/6)\)&
\(\dfrac1{6\sqrt2}\xi_{A_2}(1/6)\)\\
\(1/7\)&\(\dfrac{2^{-3/7}}7\zeta(1/7)^2\)&
\(\dfrac{2^{-1/7}}7\zeta(1/7)^2\)\\
\bottomrule
\end{tabularx}
\caption{The five positive residues.  The carrier and lower-rank factors
are derived here; Au's independent Theorems~9.4 and~9.6 supply the closed
gamma evaluations used as independent evaluations.}
\label{tab:B3C3-positive}
\end{table}
The powers of two at \(1/5,1/6,1/7\) are the products in
\eqref{eq:B3C3-singleton-coeffs} raised to the negative critical
exponent.  The closed gamma forms at \(1/4\) are those of
Au \cite[Theorems~9.4 and~9.6]{Au2025}.

The same formulas agree with the existing
\(A_3\) evaluation.  At \(s=2/5\),
\begin{equation}\label{eq:A3-Au-identity}
 \frac15\left[2B\!\left(\frac35,\frac15\right)
 +B\!\left(\frac15,\frac15\right)\right]
 =\frac{(\sqrt5+5)\Gamma(1/5)\Gamma(3/5)}
 {10\Gamma(4/5)}.
\end{equation}
Together with the identical \(1/2,1/3,1/4\) formulas, this matches
Au's Theorem~9.2.

\subsection{The shifted pair candidate at \texorpdfstring{\(1/8\)}{1/8}}
\label{subsec:B3C3-one-eighth}

The generic carrier theorem also produces a positive shifted candidate
that is absent from the final single-variable pole tables.  At
\(q=1/8\), the only incident diagonal carriers are
\[
 H_{S,1},\qquad |S|=2,
\]
for the three two-node supports.  No singleton or full-support carrier
passes through this point, and no strict support flag has length two
there.  The incidence-complete specialization formula therefore reduces
to the simple aggregate
\begin{equation}\label{eq:B3C3-one-eighth}
 [(s-\tfrac18)^{-1}]\,\xi_\Phi(s)
 =\frac18\sum_{|S|=2}
   \cR_{S,1}\!\left(\tfrac18\bfone\right),
 \qquad \Phi=B_3,C_3.
\end{equation}
Au proves that \(\xi_{B_3}\) and \(\xi_{C_3}\) are holomorphic at
\(1/8\) \cite[Theorems~9.4 and~9.6]{Au2025}.  Combined with
\eqref{eq:B3C3-one-eighth}, his result is the external exact identity forcing
both aggregate sums to vanish.  Carrier geometry alone predicts the three generic
divisors but not their diagonal cancellation; this is precisely the
distinction between generic genuineness and aggregate thinning.

\subsection{The negative-half-integer double-pole mechanism}

We first isolate the specialization zero which removes the formal
integer ladders.

\begin{lemma}[Pochhammer vanishing at nonpositive diagonal points]
\label{lem:pochhammer-vanishing}
Let \(S\) be a proper support and specialize every root exponent to
\(s=-k\), with \(k\in\Nzero\).  For a summand in the residue formula whose reduced
projective and complementary factors are holomorphic at the
specialization, nonvanishing requires
\[
 \ell\le |C(S)|k.
\]
In particular, the direct singleton term at its diagonal carrier
level \(\ell=1+N(S)k\) vanishes.
\end{lemma}

\begin{proof}
Every Taylor summand contains
\(\prod_{\alpha\in C(S)}(-k)_{k_\alpha}\).  The rising factorial
\((-k)_a\) is zero for \(a>k\); hence a nonzero composition satisfies
\(\sum_\alpha k_\alpha\le |C(S)|k\).  For a singleton, exactly one of
the \(N(S)\) meeting roots is internal, so
\(|C(S)|=N(S)-1\), whereas
\[
 1+N(S)k-(N(S)-1)k=1+k>0.
\]
The holomorphy hypothesis prevents a zero from being repaired by a pole
of a nested link.  Without that hypothesis the statement is only a
zero-order count and does not determine the full flag coefficient.
\end{proof}

\begin{lemma}[Classification of nested diagonal coincidences]
\label{lem:B3C3-collisions}
For the support counts in Table~\ref{tab:B3C3-carrier}, every strict
nested diagonal coincidence is one of the singleton--pair families
\[
\begin{array}{c|c}
N_i & (\ell_i,\ell_{ij},q)\\
\hline
5 & (1+5t,\ 2+8t,\ -t)\\
7 & (1+7t,\ 2+8t,\ -t)\\
6 & (1+3t,\ 2+4t,\ -t/2)
\end{array}
\qquad(t\in\Nzero).
\]
There is no singleton--full or pair--full coincidence, and distinct
pairs are incomparable.  At an integer coincidence the Pochhammer
product of the direct singleton term vanishes; this kills the direct
term whenever its reduced factors are regular, but a nested pole may in
principle repair that zero.  Au's theorem that every non-half-integer
pole is simple supplies the global theorem excluding such an assembled
order-two contribution.  Hence the only locations at which a
length-two flag aggregate can survive at order two are
\begin{equation}\label{eq:B3C3-negative-flags}
 (\{3\},6n+4)\subset(\{1,3\},8n+6),\qquad
 (\{3\},6n+4)\subset(\{2,3\},8n+6),
\end{equation}
at \(q_n=-(2n+1)/2\), \(n\ge0\).
\end{lemma}

\begin{proof}
For a singleton of root count \(N_i\) contained in a pair, equality of
the diagonal carriers is
\[
 \frac{1-\ell_i}{N_i}=\frac{2-\ell_{ij}}8.
\]
Solving the three linear Diophantine equations for
\(N_i=5,7,6\) gives exactly the displayed parameterizations; nonnegative
levels force \(t\ge0\).  The full-support carrier is unshifted and equals
\(1/3\).  A singleton--full equality would require
\(\ell_i=1-N_i/3\), which is either nonintegral or negative, while a
pair--full equality would require \(\ell_{ij}=-2/3\).  Hence neither
occurs.  Two distinct pairs are incomparable in the support poset and
therefore do not form a strict flag.

The first two rows and the even values of \(t\) in the third row lie
at nonpositive integers.  Lemma~\ref{lem:pochhammer-vanishing} identifies
the vanishing Pochhammer product in the direct singleton term when
the reduced factors are regular; if they are singular, the lemma alone
does not determine the assembled coefficient.  Au's
Theorems~9.4 and~9.6 supply the required global statement that no
double pole survives at these integer coincidences.  Writing
\(t=2n+1\) yields \(q_n\), \(\ell_3=6n+4\), and
\(\ell_{3j}=8n+6\).  Node three is contained in exactly the two pairs
\(\{1,3\}\) and \(\{2,3\}\), proving the final assertion.
\end{proof}

Consequently the aggregate order theorem proves
\begin{equation}\label{eq:B3C3-negative-order}
 \operatorname{pord}_{s=q_n}\xi_{B_3},\ 
 \operatorname{pord}_{s=q_n}\xi_{C_3}\le2,
\end{equation}
and every other diagonal point has order at most one.  The carrier
calculus supplies the exhaustive coincidence list and the conditional
direct Pochhammer zeros; Au's independent global theorem is the exact
theorem that rules out an assembled order-two contribution at the integer
coincidences.  Together these give the ``possibly double'' clause of his
Theorems~9.4 and~9.6.

Let \(C^{\Phi}_{13,n}\) and \(C^{\Phi}_{23,n}\) be the canonically
reduced two-flag coefficients of \eqref{eq:B3C3-negative-flags}.  Since
the diagonal derivatives of the two carrier forms are \(6\) and \(8\),
\begin{equation}\label{eq:B3C3-negative-coeff}
 [(s-q_n)^{-2}]\xi_\Phi(s)
 =\frac{C^{\Phi}_{13,n}+C^{\Phi}_{23,n}}{48}.
\end{equation}
Au evaluates the \(B_3\), \(n=0\) aggregate and proves
\begin{equation}\label{eq:B3-minus-half}
 [(s+\tfrac12)^{-2}]\xi_{B_3}(s)
 =\frac{85}{196608}\,\zeta\!\left(-\frac{13}{2}\right),
\end{equation}
so our flag normalization gives
\[
 C^{B_3}_{13,0}+C^{B_3}_{23,0}
 =\frac{85}{4096}\zeta\!\left(-\frac{13}{2}\right).
\]
This proves a genuine \(B_3\) double pole at \(-1/2\).  Au's current
paper does \emph{not} prove nonvanishing of every member of the
\(B_3\) ladder or of the \(C_3\) ladder.  Those questions remain the
explicit obstruction to upgrading \eqref{eq:B3C3-negative-order} from
an order bound to a family theorem.

\section{Diagonal specialization}

On the diagonal \(s_\alpha=s\),
\[
 \Lambda_{S,\ell}(s\bfone)
 =N(S)\left(s-\frac{|S|-\ell}{N(S)}\right).
\]
For a single transverse divisor, the diagonal residue contribution is
\[
 \frac1{N(S)}
 \cR_{S,\ell}\!\left(\frac{|S|-\ell}{N(S)}\bfone\right).
\]
When several divisors meet the diagonal, one must use
Theorem~\ref{thm:aggregate-order}: generic divisor survival does not
imply survival of the aggregate diagonal coefficient.

The diagonal can have additional cancellations coming from symmetries of
the central-lattice model.  Relating those block symmetries to the
standard-support coefficients \(C_{\mathfrak F}\) requires a separate
comparison theorem and is not needed here.  We therefore determine all
diagonal Laurent coefficients directly from the incidence-complete
formula below.

\subsection{Diagonal Laurent coefficients}

Let \(q\) be a diagonal carrier and let \(\mathfrak F\) range over all
incident contributing flags with
\(\Lambda_w(s\bfone)=N(w)(s-q)\).  The coefficient of
\((s-q)^{-m}\) is
\begin{equation}\label{eq:diagonal-aggregate}
 [(s-q)^{-m}]Z_\Phi(s\bfone)
 =\sum_{|\mathfrak F|\ge m}
 \frac{
  \left.\dfrac{d^{|\mathfrak F|-m}}{ds^{|\mathfrak F|-m}}
  C_{\mathfrak F}(s\bfone)\right|_{s=q}}
 {(|\mathfrak F|-m)!\prod_{w\in\mathfrak F}N(w)}.
\end{equation}
The pole order is the largest \(m\) for which the complete sum in
\eqref{eq:diagonal-aggregate} is nonzero.  Formula
\eqref{eq:diagonal-aggregate}, not the largest individual flag order,
is the diagonal dictionary.

Selected diagonal specializations are listed in Table~\ref{tab:diagonal-certificate}:
\begin{table}[ht]
\centering
\small
\begin{tabularx}{\textwidth}{@{}cclX@{}}
\toprule
Type & pole & order & diagonal principal coefficient\\
\midrule
\(A_2\)&\(2/3\)&1&\(\Gamma(1/3)^3/(2\sqrt3\,\pi)\)\\
\(A_2\)&\(1/2\)&1&\(\zeta(1/2)\)\\
\(A_3\)&\(1/2\)&1&\(\Gamma(1/4)^4/(24\pi)\)\\
\(A_3\)&\(2/5\)&1&
\(\frac15[2B(3/5,1/5)+B(1/5,1/5)]\zeta(2/5)\)\\
\(A_3\)&\(1/3\)&1&\(\frac23\xi_{A_2}(1/3)\)\\
\(A_3\)&\(1/4\)&1&\(\frac14\zeta(1/4)^2\)\\
\(C_2\)&\(1/3\)&1&
\(\frac13(1+2^{-1/3})\zeta(1/3)\)\\
\(G_2\)&\(1/5\)&1&
\(\frac15(18^{-1/5}+2^{-1/5})\zeta(1/5)\)\\
\(A_4\)&\(1/4\)&2&\(\zeta(1/4)^2/16\)\\
\bottomrule
\end{tabularx}
\caption{Representative diagonal coefficients.}
\label{tab:diagonal-certificate}
\end{table}

\begin{table}[ht]
\centering
\small
\begin{tabular}{@{}cccl@{}}
\toprule
Type & Coxeter \(h\) & leading point \(2/h\) & closed gamma form\\
\midrule
\(A_2\)&3&\(2/3\)&\(\Gamma(1/3)^3/(2\sqrt3\,\pi)\)\\
\(B_2=C_2\)&4&\(1/2\)&\(\Gamma(1/4)^2/(8\sqrt{2\pi})\)\\
\(G_2\)&6&\(1/3\)&\(\Gamma(1/3)^3/(2^{8/3}3^{3/2}\pi)\)\\
\(A_3\)&4&\(1/2\)&\(\Gamma(1/4)^4/(24\pi)\)\\
\(B_3,C_3\)&6&\(1/3\)&\(\Gamma(1/3)^6/(96\pi^2)\)\\
\(A_4\)&5&\(2/5\)& universal leading-residue formula\\
\bottomrule
\end{tabular}
\caption{Leading diagonal residues used in the low-rank comparisons.  The first
five gamma forms are independently present in Au's calculations; the
all-type equality is also supplied by the universal leading-residue
theorem.}
\label{tab:Au-leading-comparisons}
\end{table}

\section{Concluding remarks}

The carrier arrangement records the possible singular geometry, but it
does not determine the order after specialization.  The decisive
object is the aggregate Laurent coefficient of the incident strict
flags.  This distinction accounts simultaneously for Pochhammer
zeros, disconnected gamma factors, and cancellation among several
carrier components.

Several questions remain.  The general flag formula does not by
itself classify the nonzero coefficients at every intersection.
Likewise, the negative-half-integer \(B_3\) and \(C_3\) ladders are
known to have order at most two, but only the \(B_3\) coefficient at
\(-1/2\) has presently been evaluated as nonzero.  It would also be
useful to extend the explicit low-rank analysis to the remaining
exceptional types and to other KMT variants.

The results concern the untwisted strongly dominant KMT function in
the convention of Section~2.  Weyl-symmetrized, exponentially twisted,
and epsilon-decorated variants require separate bookkeeping.  No
statement is made here about the distribution of nontrivial zeros of
the complementary zeta factors.
\section*{Computational verification and declarations}

Symbolic computation was used during preparation to check finite
Taylor expansions, support counts, simplex Jacobians, gamma-factor
normalizations, and the displayed low-rank identities.  The arguments
in the paper are independent of numerical approximation.

\paragraph{AI-assisted drafting disclosure.}
AI tools were used for mathematical exploration, literature search,
proof checking, and drafting.  The author reviewed the manuscript and
accepts responsibility for its content.

\paragraph{Data availability.}
The \LaTeX\ source is supplied with the manuscript.


\begingroup
\footnotesize
\begin{thebibliography}{99}


\bibitem{Au2025}
K.~C.~Au.
\newblock On single-variable Witten zeta functions of rank two and three.
\newblock arXiv:2412.17196, version~3, revised 14 November 2025.

\bibitem{BudurShiZuo2025}
N.~Budur, Q.~Shi, and H.~Zuo.
\newblock Polar loci of multivariable Archimedean zeta functions.
\newblock arXiv:2504.10051, 2025.

\bibitem{BEL2007}
G.~Bhowmik, D.~Essouabri, and B.~Lichtin.
\newblock Meromorphic continuation of multivariable Euler products.
\newblock {\em Forum Mathematicum}, 19(6):1111--1139, 2007.
\newblock DOI: 10.1515/FORUM.2007.044.

\bibitem{AET2001}
S.~Akiyama, S.~Egami, and Y.~Tanigawa.
\newblock Analytic continuation of multiple zeta-functions and their values at
  non-positive integers.
\newblock {\em Acta Arithmetica}, 98(2):107--116, 2001.

\bibitem{DeConciniProcesi1995}
C.~De~Concini and C.~Procesi.
\newblock Wonderful models of subspace arrangements.
\newblock {\em Selecta Mathematica (N.S.)}, 1(3):459--494, 1995.


\bibitem{FeichtnerKozlov2004}
E.-M.~Feichtner and D.~N.~Kozlov.
\newblock Incidence combinatorics of resolutions.
\newblock {\em Selecta Mathematica (N.S.)}, 10(1):37--60, 2004.

\bibitem{Essouabri1997}
D.~Essouabri.
\newblock Singularit\'e de s\'eries de {Dirichlet} associ\'ees \`a des
  polyn\^omes de plusieurs variables et applications en th\'eorie analytique
  des nombres.
\newblock {\em Annales de l'Institut Fourier}, 47(2):429--483, 1997.

\bibitem{KMT2010PLMS}
Y.~Komori, K.~Matsumoto, and H.~Tsumura.
\newblock On multiple Bernoulli polynomials and multiple {$L$}-functions of
  root systems.
\newblock {\em Proceedings of the London Mathematical Society},
  100(2):303--347, 2010.

\bibitem{KMT2010II}
Y.~Komori, K.~Matsumoto, and H.~Tsumura.
\newblock On Witten multiple zeta-functions associated with semisimple Lie
  algebras {II}.
\newblock {\em Journal of the Mathematical Society of Japan}, 62(2):355--394,
  2010.

\bibitem{KMT2011G2}
Y.~Komori, K.~Matsumoto, and H.~Tsumura.
\newblock On Witten multiple zeta-functions associated with semisimple Lie
  algebras {IV}.
\newblock {\em Glasgow Mathematical Journal}, 53(1):185--206, 2011.


\bibitem{KMTBook2023}
Y.~Komori, K.~Matsumoto, and H.~Tsumura.
\newblock {\em The Theory of Zeta-Functions of Root Systems}.
\newblock Springer Monographs in Mathematics. Springer, Singapore, 2023.
\newblock DOI: 10.1007/978-981-99-0910-0.


\bibitem{GPZ2014}
L.~Guo, S.~Paycha, and B.~Zhang.
\newblock Conical zeta values and their double subdivision relations.
\newblock {\em Advances in Mathematics}, 252:343--381, 2014.
\newblock DOI: 10.1016/j.aim.2013.10.022.

\bibitem{GPZ2020}
L.~Guo, S.~Paycha, and B.~Zhang.
\newblock A conical approach to Laurent expansions for multivariate
 meromorphic germs with linear poles.
\newblock {\em Pacific Journal of Mathematics}, 307(1):159--196, 2020.
\newblock DOI: 10.2140/pjm.2020.307.159.

\bibitem{KMT2012III}
Y.~Komori, K.~Matsumoto, and H.~Tsumura.
\newblock On Witten multiple zeta-functions associated with semisimple Lie
 algebras {III}.
\newblock In D.~Bump, S.~Friedberg, and D.~Goldfeld, editors,
 {\em Multiple Dirichlet Series, L-functions and Automorphic Forms},
 volume 300 of {\em Progress in Mathematics}, pages 223--286.
 Birkh\"auser/Springer, 2012.

\bibitem{Lichtin1988}
B.~Lichtin.
\newblock Generalized Dirichlet series and \(b\)-functions.
\newblock {\em Compositio Mathematica}, 65(1):81--120, 1988.

\bibitem{Lopez2022}
V.~L\'opez.
\newblock Pole structure of Shintani zeta functions and Newton polytopes.
\newblock arXiv:2205.15620, 2022.

\bibitem{MOS2024}
K.~Matsumoto, K.~Onodera, and S.~Sahoo.
\newblock Mordell--Tornheim multiple zeta-functions, their integral analogues,
 and relations among multiple polylogarithms.
\newblock arXiv:2409.19980, 2024.

\bibitem{Li2009}
L.~Li.
\newblock Wonderful compactification of an arrangement of subvarieties.
\newblock {\em Michigan Mathematical Journal}, 58(2):535--563, 2009.

\bibitem{MNOST2008}
K.~Matsumoto, T.~Nakamura, H.~Ochiai, and H.~Tsumura.
\newblock On value-relations, functional relations and singularities of
  Mordell--Tornheim and related triple zeta-functions.
\newblock {\em Acta Arithmetica}, 132(2):99--125, 2008.



\bibitem{Rutard2023}
S.~Rutard.
\newblock Values and derivative values at nonpositive integers of generalized
 multiple Hurwitz zeta functions, applications to Witten zeta functions.
\newblock arXiv:2312.04725, 2023.


\bibitem{Matsumoto2002}
K.~Matsumoto.
\newblock On the analytic continuation of various multiple zeta-functions.
\newblock In M.~A.~Bennett et al., editors, {\em Number Theory for the
 Millennium II}, pages 417--440. A K Peters, 2002.

\bibitem{Onozuka2013}
T.~Onozuka.
\newblock Analytic continuation of multiple zeta-functions and the
 asymptotic behavior at non-positive integers.
\newblock {\em Functiones et Approximatio Commentarii Mathematici},
 49(2):331--348, 2013.
\newblock DOI: 10.7169/facm/2013.49.2.11.

\bibitem{KOOT2022}
S.-y.~Kadota, T.~Okamoto, M.~Ono, and K.~Tasaka.
\newblock On a unified double zeta function of Mordell--Tornheim type.
\newblock {\em Lithuanian Mathematical Journal}, 62(2):207--217, 2022.
\newblock DOI: 10.1007/s10986-022-09556-x.

\bibitem{Zhao2000}
J.~Zhao.
\newblock Analytic continuation of multiple zeta functions.
\newblock {\em Proceedings of the American Mathematical Society},
 128(5):1275--1283, 2000.
\newblock DOI: 10.1090/S0002-9939-99-05398-8.

\bibitem{MatuzasLeading2026}
J.~Matuzas.
\newblock A universal leading-residue formula for {Witten} zeta functions.
\newblock arXiv:2607.12728, 2026.

\bibitem{MatsumotoTsumura2006}
K.~Matsumoto and H.~Tsumura.
\newblock On Witten multiple zeta-functions associated with semisimple Lie
  algebras {I}.
\newblock {\em Annales de l'Institut Fourier}, 56(5):1457--1504, 2006.

\end{thebibliography}
\endgroup
\end{document}